\documentclass{article}
\pdfoutput=1

\usepackage{framed,multirow}
\usepackage{epsfig}

\usepackage{amssymb,amsmath,amsthm}
\usepackage{latexsym}

\usepackage{url}
\usepackage{xcolor}
\definecolor{newcolor}{rgb}{.8,.349,.1}

\usepackage{hyperref}

\newcommand{\Lag}{\mathrm Lag}
\newcommand{\entreguillemets}[1]{``{#1}''} 

\newtheorem{observation}{Observation}
\newtheorem{theorem}{Theorem}
\newtheorem{definition}{Definition}

\begin{document}



\title{Notions of optimal transport theory and how to implement them on a computer}

\author{Bruno L\'{e}vy and Erica Schwindt}



\maketitle

\begin{abstract}
  This article gives an introduction to optimal transport, a mathematical
theory that makes it possible to measure distances between
functions (or distances between more general objects), to
interpolate between objects or to enforce mass/volume conservation
in certain computational physics simulations. Optimal transport
is a rich scientific domain, with active research communities, both on its theoretical
aspects and on more applicative considerations, such as geometry
processing and machine learning. This article aims at explaining
the main principles behind the theory of optimal transport, introduce
the different involved notions, and more importantly, how they
relate, to let the reader grasp an intuition of the elegant
theory that structures them. Then we will consider a specific setting,
called semi-discrete, where a continuous function is transported to a discrete
sum of Dirac masses. Studying this specific setting naturally leads to an
efficient computational algorithm, that uses classical notions of computational
geometry, such as a generalization of Voronoi diagrams called Laguerre diagrams.
\end{abstract}




\section{Introduction}
This article presents an introduction to optimal transport.
It summarizes and complements a series of conferences given by B. L\'evy between
2014 and 2017. 
The presentations stays at an elementary level, that corresponds to a computer scientist's
vision of the problem. In the article, we stick to using standard notions of analysis
(functions, integrals) and linear algebra (vectors, matrices), and give an intuition of
the notion of measure. The main objective of the presentation is to understand the overall
structure of the reasoning \footnote{Teach principles, not equations. [R. Feynman]}, and
to follow a continuous path from the theory to an efficient
algorithm that can be implemented in a computer. \\


Optimal transport, initially studied by Monge, \cite{Monge1784}, is a
very general mathematical framework that can be used to model a wide
class of application domains. In particular, it is a natural
formulation for several fundamental questions in computer graphics
\cite{DBLP:journals/focm/Memoli11,DBLP:journals/cgf/Merigot11,DBLP:journals/tog/BonneelPPH11},
because it makes it possible to define new ways of \emph{comparing} functions,
of measuring \emph{distances} between functions and
\emph{interpolating} between two (or more) functions~:

\begin{figure}
      \centerline{
         \includegraphics[width=\columnwidth]{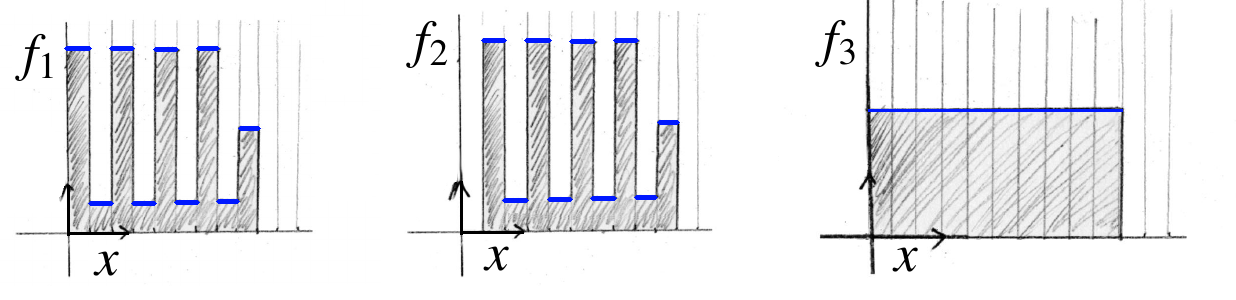}
      }
      \caption{
        Comparing functions: one would like to say that $f_1$ is nearer to $f_2$ than $f_3$, but
        the classical $L_2$ norm \entreguillemets{does not see} that the graph of $f_2$ corresponds
        to the graph of $f_1$ slightly shifted along the $x$ axis.
      }
      \label{fig:distance_fonction}
\end{figure}  

\paragraph{Comparing functions} Consider the functions $f_1$, $f_2$ and $f_3$ in Figure \ref{fig:distance_fonction}.
Here we have chosen a function $f_1$ with a wildly oscillating graph, and a function $f_2$ obtained by translating the graph
of $f_1$ along the $x$ axis. The function $f_3$ corresponds to the mean value
of $f_1$ (or $f_2$). If one measures the relative distances between these functions using the classical $L_2$ norm,
that is  $d_{L_2}(f,g) = \int (f(x) - g(x))^2 dx$, one will find that $f_1$ is nearer to $f_3$ than $f_2$. Optimal transport
makes it possible to define a distance that will take into account that the graph of $f_2$ can be obtained from $f_1$ through
a translation (like here), or through a deformation of the graph of $f_1$. From the point of view of this new distance,
the function $f_1$ will be nearer to $f_2$ than to $f_3$.

\begin{figure}
      \centerline{
         \includegraphics[width=\columnwidth]{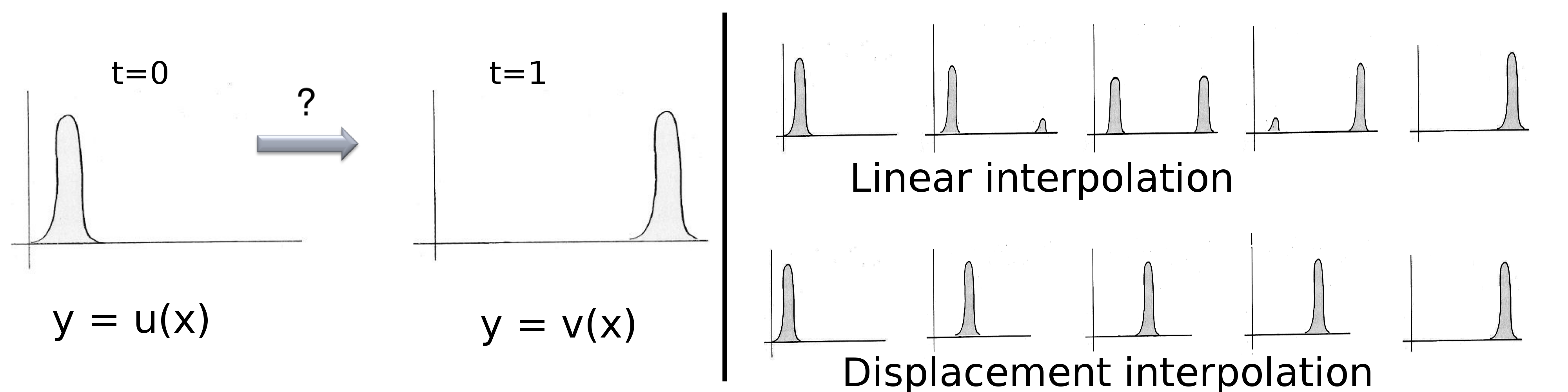}
      }
      \caption{
        Interpolating between two functions: linear interpolation makes a hump disappear
        while the other hump appears; displacement interpolation, stemming from optimal transport,
        will displace the hump as expected.
      }
      \label{fig:interpolation}
\end{figure}  

\paragraph{Interpolating between functions:} Now consider the functions $u$ and $v$
in Figure \ref{fig:interpolation}. Here we suppose that $u$ corresponds
to a certain physical quantity measured at an initial time $t=0$ and that $v$
corresponds to the same phenomenon measured at a final time $t=1$. The problem
that we consider now consists in reconstructing what happened between $t=0$ and
$t=1$. If we use linear interpolation (Figure \ref{fig:interpolation}, top-right),
we will see the left hump progressively disappearing while the right hump progressively
appears, which is not very realistic if the functions represent for instance a propagating
wave. Optimal transport makes it possible to define another type of interpolation
(Mc. Cann's displacement interpolation, Figure
\ref{fig:interpolation}, bottom-right), that will progressively displace and
morph the graph of $u$ into the graph of $v$. \\

Optimal transport makes it possible to define a \emph{geometry} of
a space of functions\footnote{or more general objects, called probability
measures, more on this later.}, and thus gives a definition of \emph{distance}
in this space, as well as means of interpolating between different functions,
and in general, defining the barycenter of a weighted family of functions, in
a very general context. Thus, optimal transport appears as a fundamental tool in many
applied domains. In computer graphics, applications were proposed, to compare and
interpolate objects of diverse natures \cite{DBLP:journals/tog/BonneelPPH11}, to generate
lenses that can concentrate light to form caustics in a prescribed manner
\cite{DBLP:journals/corr/abs-1708-04820,
  DBLP:journals/tog/SchwartzburgTTP14}. Moreover, optimal transport
defines new tools that can be used to discretize Partial Differential
Equations, and define new numerical solution mechanisms \cite{JKO2}.
This type of numerical solution mechanism can be used to simulate
for instance fluids \cite{Gallouet2017}, with spectacular
applications and results in computer graphics
\cite{DBLP:journals/tog/GoesWHPD15}. \\

\label{chap5:sec:ot:OT}

The two sections that follow are partly inspired by \cite{OTON}, \cite{OTintro}, \cite{MAEintro} and \cite{OTuserguide},
but stay at an elementary level. Here the main goal is to give an intuition of the different concepts,
and more importantly an idea of the way the relate together. Finally we will see how they can be
directly used to design a computational algorithm with very good performance, that can be used in practice
in several application domains.


\begin{figure}
      \centerline{
         \includegraphics[width=\columnwidth]{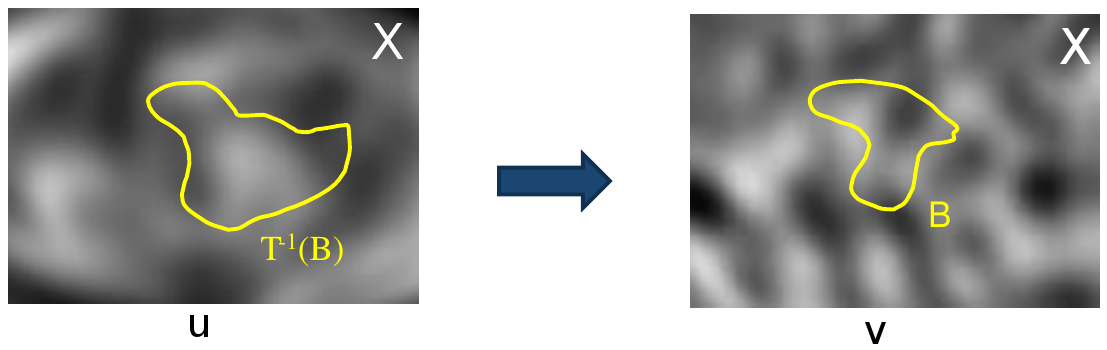}
      }
      \caption{
        Given two terrains defined by their height functions $u$ and $v$, symbolized here
        as gray levels, Monge's problem consists in transforming one terrain into the other one
        by moving matter through an application $T$. This application needs to satisfy a mass
        conservation constraint.
      }
      \label{fig:monge}
\end{figure}  

\section{Monge's problem}

The initial optimal transport problem was first introduced and
studied by Monge, right before the French revolution \cite{Monge1784}.
We first give an intuitive idea of the problem, then quickly introduce
the notion of \emph{measure}, that is necessary to formally state the
problem in its most general form and to analyze it.

\subsection{Intuition} Monge's initial motivation to study this problem
was very practical: supposing you have an army of workmen, how can you
transform a terrain with an initial landscape into a given desired target
landscape, while minimizing the total amount of work ?

Monge's initial problem statement was as follows:
$$
\begin{array}{l}
   \inf\limits_{T:  X \rightarrow X} \int\limits_X c(x, T(x)) u(x) dx \\[6mm]
   \mbox{\ \ \ \  subject to: } \\[6mm]
   \forall B \subset X, \int\limits_{T^{-1}(B)} u(x)dx = \int\limits_B v(x)dx
\end{array}   
$$
where $X$ is a subset of ${\mathbb R}^2$, $u$ and $v$ are two positive functions
defined on $X$ and such that $\int_X u(x) dx$ = $\int_Y v(x)dx$, and $c(\cdot,\cdot)$ is
a convex distance (the Euclidean distance in Monge's initial problem statement). 

The functions $u$ and $v$ represent the height of the current landscape and the height
of the target landscape respectively (symbolized as gray levels in Figure \ref{fig:monge}).
The problem consists in finding (if it exists) a function $T$ from $X$ to $X$ that
transforms the current landscape $u$ into the desired one $v$, while minimizing the product
of the amount of transported earth $u(x)$ with the distance $c(x,T(x))$ to which it was transported.
Clearly, the amount of earth is conserved during transport, thus the total quantity of earth should
be the same in the source and target landscapes (the integrals of $u$ and $v$ over $X$ should coincide).
This \emph{global} matter conservation constraint needs to be completed with a \emph{local} one. The local
matter conservation constraint enforces that in the target landscape, the quantity of earth received in
any subset $B$ of $X$ corresponds to what was transported here, that is the quantity of earth initially
present in the pre-image $T^{-1}(B)$ of $B$ under $T$. Without this constraint, one could locally create
matter in some places and annihilate matter in other places in a counterbalancing way. A map $T$
that satisfies the local mass conservation constraint is called a \emph{transport map}.

\begin{figure}
      \centerline{
         \includegraphics[width=5cm]{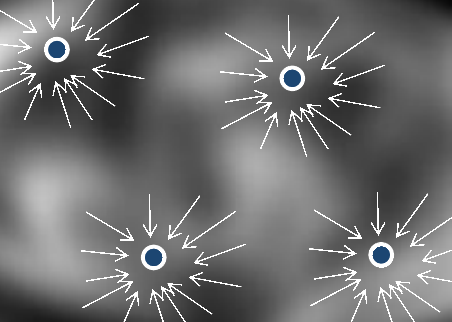}
      }
      \caption{
         Transport from a function (gray levels) to a discrete point-set
         (blue disks).
      }
      \label{fig:discrete1}
\end{figure}  

\subsection{Monge's problem with measures}
\label{sect:monge_measure}
We now suppose that instead of a \entreguillemets{target landscape}, we wish to
transport earth (or a resource) towards a set of points (that will be denoted by $Y$ for now on), that
represent for instance a set of factories that exploit a resource, see 
Figure \ref{fig:discrete1}. Each factory wishes to receive a certain quantity of resource
(depending for instance of the number of potential customers around the factory). Thus, the
function $v$ that represents the \entreguillemets{target landscape} is replaced with a function
on a finite set of points. However, if a function $v$ is zero everywhere except on a finite set
of points, then its integral over $X$ is also zero. This is a problem, because for instance
one cannot properly express the mass conservation constraint. For this reason, the notion
of function is not rich enough for representing this configuration. One can use instead
\emph{measures} (more on this below), and associate with each factory a \emph{Dirac mass} weighted by the quantity
of resource to be transported to the factory. 

From now on, we will use measures $\mu$ and $\nu$ to represent the
\entreguillemets{current landscape} and the \entreguillemets{target
  landscape}.  These measures are supported by sets $X$ and $Y$, that
may be different sets (in the present example, $X$ is a subset of
${\mathbb R}^2$ and $Y$ is a discrete set of points). Using measures
instead of function not only makes it possible to study our
\entreguillemets{transport to discrete set of factories} problem, but
also it can be used to formalize computer objects (meshes) and
directly leads to a computational algorithm. This algorithm is very
elegant because it is a \emph{verbatim computer translation of the
  mathematical theory} (see \S\ref{sect:algo}).  In this particular
setting, translating from the mathematical language to the algorithmic
setting does not require to make any approximation. This is made
possible by the generality of the notion of measure.

The reader who wishes to learn more on measure theory may refer to
the textbook \cite{TaoMeasureTheory}. To keep the length
of this article reasonable, we will not give here the formal definition
of a measure. In our context, one can think of a measure as a
\entreguillemets{function} that can be only queried using integrals
and that can be \entreguillemets{concentrated} on very small sets (points).
The following table can be used to intuitively translate from the
\entreguillemets{language of functions} to the \entreguillemets{language of measures}~:
$$
\begin{array}{|c|c|}
    \hline
    \mbox{Function $u$}    & \mbox{Measure $\mu$ } \\[2mm]
    \hline    
    \int_B u(x) dx      & \mu(B) \mbox{ or } \int_B d\mu \\[2mm]
    \int_B f(x) u(x) dx & \int_B f(x) d\mu \\[2mm]
    u(x)                & N/A \\
    \hline    
\end{array}
$$
(Note: in contrast with functions, measures cannot be evaluated at a point,
 they can be only integrated over domains). \\

In its version with measures, Monge's problem can be stated as follows:

\begin{equation}
  \begin{array}{l}
    \inf\limits_{T: X\rightarrow Y} \int\limits_X c(x,T(x)) d\mu
    \  \mbox{ subject to }\ \nu = T\sharp\mu
    \end{array}
    \label{eqn:monge} \tag{M}
\end{equation}
where $X$ and $Y$ are Borel sets (that is, sets that can be measured),
$\mu$ and $\nu$ are two measures on $X$ and $Y$ respectively such that $\mu(X)
= \nu(Y)$ and $c(\cdot,\cdot)$ is a convex distance.  The constraint $\nu =
T\sharp\mu$, that reads \entreguillemets{$T$ pushes $\mu$ onto $\nu$}
corresponds to the local mass conservation constraint. Given a
measure $\mu$ on $X$ and a map $T$ from $X$ to
$Y$, the measure $T\sharp\mu$ on $Y$, called \entreguillemets{the
pushforward of $\mu$ by $T$}, is such that $T\sharp\mu(B)
= \mu(T^{-1}(B))$ for all Borel set $B \subset Y$. Thus, the local mass
conservation constraint means that $\mu(T^{-1}(B))
= \nu(B)$ for all Borel set $B$ $\subset$ $Y$. \\

The local mass conservation constraint makes the problem very
difficult: imagine now that you want to implement a computer
program that enforces it: the constraint concerns
\emph{all the subsets} $B$ of $Y$. Could you imagine an algorithm
that just \emph{tests} whether a given map satisfies it ?
What about \emph{enforcing} it ? We will see below a series of
transformations of the initial problem into equivalent problems,
where the constraint becomes \emph{linear}. We will finally end up
with a simple convex optimization problem, that can be solved
numerically using classical methods. \\

\begin{figure}
      \centerline{
         \includegraphics[width=10cm]{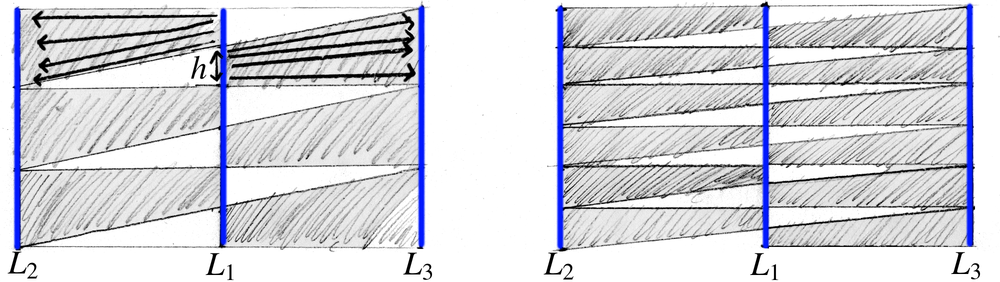}
      }
      \caption{A classical example of the existence problem: there is no optimal transport between
        a segment $L_1$ and two parallel segments $L_2$ and $L_3$
      (it is always possible to find a better transport by replacing $h$ with $h/2$).      
      }
      \label{fig:monge_pb}
\end{figure}  

\begin{figure*}
\includegraphics[width=\textwidth]{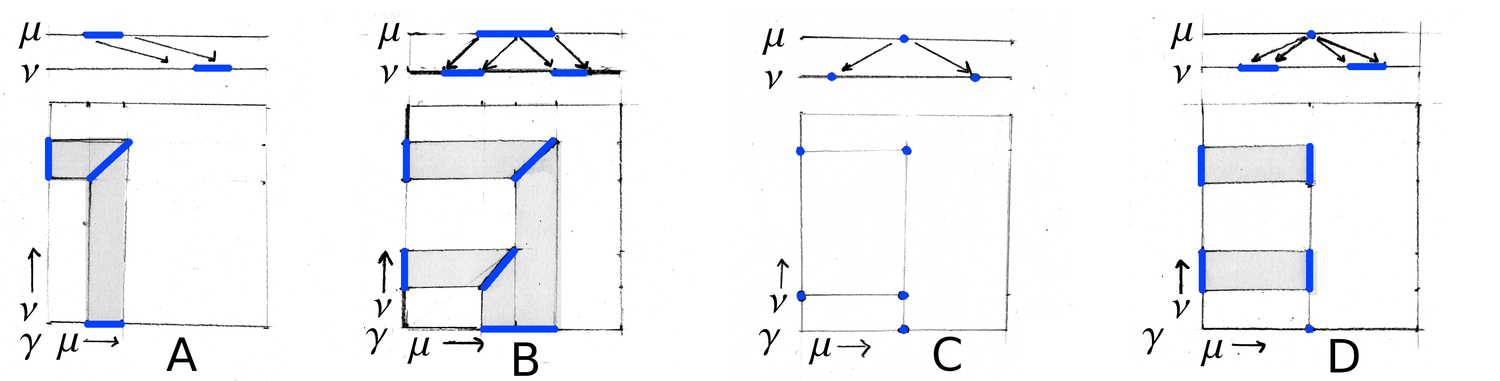}
\caption{
  Four example of transport plans in 1D. A: a segment is
  translated. B: a segment is split into two segments.
  C: a Dirac mass is split into two Dirac masses; D: a Dirac
  mass is spread along two segments. The first two examples
  (A and B) have the form $(Id \times T)\sharp \mu$ where $T$ is a transport
  map. The third and fourth ones (C and D) have no corresponding transport map,
  because each of them splits a Dirac mass.} 
\label{fig:ch5:ot:OTplans}
\end{figure*}

Before then, let us get back to examine the original problem. The local mass conservation
constraint is not the only difficulty: the functional optimized by Monge's problem is
non-symmetric, and this causes additional difficulties when studying the existence
of solutions for problem \eqref{eqn:monge}. The problem is not symmetric because $T$ needs to
be a map, therefore the source and target landscape do not play the same role.
Thus, it is possible to merge earth (if $T(x_1) = T(x_2)$ for two different points
$x_1$ and $x_2$), but it is not possible to split earth (for that, we would need a
\entreguillemets{map} $T$ that could send the same point $x$ to two different points $y_1$ and $y_2$).
The problem is illustrated in Figure \ref{fig:monge_pb}:
suppose that you want to compute the optimal transport between a segment
$L_1$ (that symbolizes a \entreguillemets{wall of earth}) and two parallel
segments $L_2$ and $L_3$ (that symbolize two \entreguillemets{trenches} with a depth that
correspond to half the height of the wall of earth). Now we want to transport the wall of earth
to the trenches, to make the landscape flat. To do so, it is possible to decompose $L_1$ into
segments of length $h$, sent alternatively towards $L_2$ and $L_3$
(Figure \ref{fig:monge_pb} on the left). For any length
$h$, it is always possible to find a better map $T$, that is a lower value
of the functional in \eqref{eqn:monge}, by subdividing $L_1$ into
smaller segments (Figure \ref{fig:monge_pb} on the right). The
best way to proceed consists in sending from each point of
$L_1$ \emph{half} the earth to $L_2$ and \emph{half} the earth to $L_3$, which
cannot be represented by a map. Thus, the best solution of problem \eqref{eqn:monge} is not a
map. In a more general setting, this problem appears each time the source measure
$\mu$ has mass concentrated on a manifold of dimension $d-1$ \cite{McCannEUMMP05}
(like the segment $L_1$ in the present example).

\section{Kantorovich's relaxed problem}

To overcome this difficulty, Kantorovich stated a problem with a
larger space of solutions, that is, a relaxation of problem \eqref{eqn:monge},
where mass can be both split and merged. The idea consists
in solving for the \entreguillemets{graph of $T$} instead of $T$.
One may think of the graph of $T$ as a function $g$ defined on $X \times Y$
that indicates for each couple of points $x\in X, y\in Y$ how much matter goes from
$x$ to $y$. However, once again, we cannot use standard functions to represent the
graph of $T$: if you think about the graph of a univariate function $x \mapsto f(x)$,
it is defined on $\mathbb{R}^2$ but \emph{concentrated} on a curve. For this reason, as in our
previous example with factories, one needs to use measures. Thus, we are now 
looking for a measure $\gamma$ supported by the
product space $X \times Y$. The relaxed problem is stated as follows:

\begin{equation}
   \begin{array}{l}
     \inf\limits_{\gamma} \left\{ \int\limits_{X \times Y} c(x,y) d \gamma \ | \ \gamma \ge 0 \mbox{ and } \gamma \in \Pi(\mu,\nu) \right\} \\[5mm]
   \mbox{where: } \\[3mm]
     \Pi(\mu,\nu) = \{ \gamma \in {\cal P}(X \times Y) \ | \
     (P_X)\sharp\gamma = \mu \ ; \ (P_Y)\sharp\gamma = \nu \}
   \end{array}
   \label{eqn:K} \tag{K}
\end{equation}
where $(P_X)$ and $(P_Y)$ denote the two projections $(x,y) \in X \times Y \mapsto x$ and
$(x,y) \in X \times Y \mapsto y$ respectively.

The two measures $(P_X)\sharp\gamma$ and $(P_Y)\sharp\gamma$ obtained by pushing forward $\gamma$ by the two projections
are called the \emph{marginals} of $\gamma$. The measures $\gamma$ in the admissible set $\Pi(\mu,\nu)$,
that is, the measures that have $\mu$ and $\nu$ as marginals, are called \emph{optimal transport plans}. Let us now have a
closer look at the two constraints on the marginals $(P_X)\sharp\gamma = \mu$ and $(P_X)\sharp\gamma$
that define the set of optimal transport plans $\Pi(\mu,\nu)$. Recalling the definition of the pushforward (previous subsection),
these two constraints can also be written as:
\begin{equation}
   \begin{array}{lcl}
       (P_X)\sharp\gamma = \mu & \iff & \forall B \subset X, \int_B d\mu = \int_{B \times Y} d\gamma \\[2mm]
       (P_Y)\sharp\gamma = \nu & \iff & \forall B^\prime \subset Y, \int_{B^\prime} d\nu = \int_{X \times B^\prime} d\gamma.       
   \end{array}
\label{eqn:ch5:ot:marginales}   
\end{equation}
Intuitively, the first constraint $(P_X)\sharp\gamma = \mu$ means that everything that comes from a subset $B$
of $X$ should correspond to the amount of matter (initially) contained by $B$ in the source landscape, and the second one
$(P_Y)\sharp\gamma = \nu$ means that everything that goes into a subset $B^{\prime}$ of $Y$ should correspond to the
(prescribed) amount of matter contained by $B^{\prime}$ in the target landscape $\nu$.

\begin{figure*}
      \centerline{
         \includegraphics[width=\textwidth]{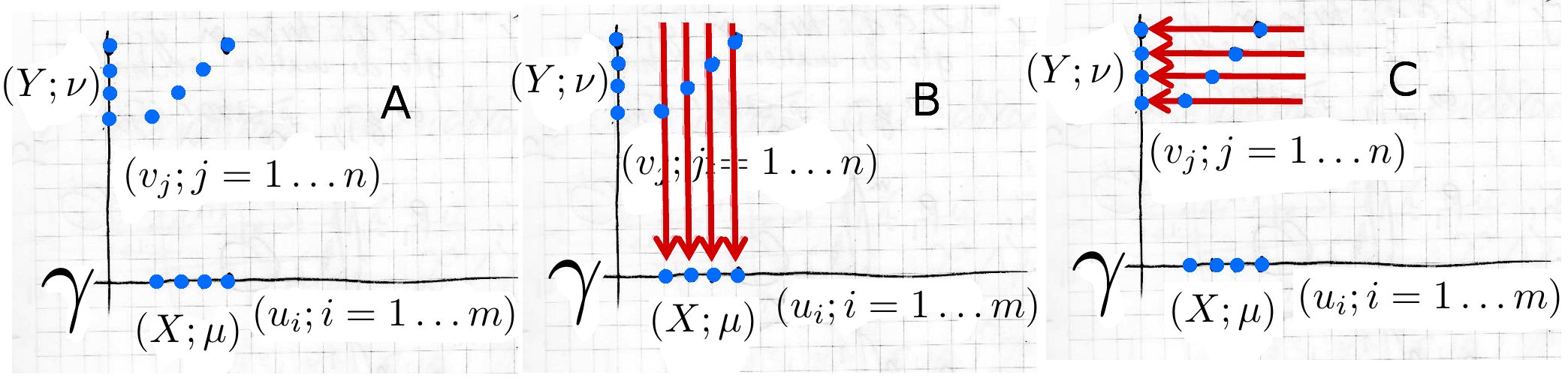}
      }
      \caption{A discrete version of Kantorovich's problem.}
      \label{fig:DK}
\end{figure*}  

We now examine the relation between the relaxed problem \eqref{eqn:K}
and the initial problem \eqref{eqn:monge}. One can easily check that among the optimal
transport plans, those with the form $(Id \times T)\sharp \mu$
correspond to a transport map $T$:

\begin{observation}
  If $(Id \times T)\sharp\mu$ is a transport plan, then $T$ pushes $\mu$ onto $\nu$.
\end{observation}
\begin{proof}
 $(Id \times T)\sharp\mu$ is in $\Pi(\mu,\nu)$, thus
  $(P_Y)\sharp(Id \times T)\sharp \mu = \nu$, or
  $\left((P_Y)\circ(Id \times T)\right)\sharp \mu = \nu$, and finally
  $T\sharp\mu = \nu$.
\end{proof}

We can now observe that if a transport plan $\gamma$ has the form $\gamma = (Id \times T)\sharp \mu$, then
problem \eqref{eqn:K} becomes:
$$
     \mbox{min} \left\{ \int\limits_{X \times Y} c(x,y)d\left( (Id \times T)\sharp \mu \right) \right\} 
\quad = \quad
     \mbox{min} \left\{ \int\limits_X c(x,T(x)) d \mu \right)
$$
(one retrieves the initial Monge problem). \\

To further help grasping an intuition of this notion of transport plan,
we show four 1D examples in Figure \ref{fig:ch5:ot:OTplans} (the transport
plan is then $1D \times 1D = 2D$). Intuitively, the transport plan $\gamma$
may be thought of as a \entreguillemets{table} indexed by  $x$ and $y$
that indicates the quantity of matter transported from $x$ to $y$.
More exactly\footnote{We recall that one cannot evaluate a measure $\gamma$
at a point $(x,y)$, we can just compute integrals with $\gamma$.},
the measure $\gamma$ is non-zero on subsets of $X \times Y$ that contain
points $(x,y)$ such that some matter is transported from $x$ to $y$. Whenever
$\gamma$ derives from a transport map $T$, that is if $\gamma$ has the form
$(Id \times T)\sharp \mu$, then we can consider $\gamma$ as the
\entreguillemets{graph of $T$} like in the first two examples of Figure \ref{fig:ch5:ot:OTplans} (A) and (B)\footnote{Note that
  the measure $\mu$ is supposed to be absolutely continuous with respect to the
  Lebesgue measure. This is required, because for instance in example (B) of Figure \ref{fig:ch5:ot:OTplans},
  the transport map $T$ is undefined at the center of the segment. The absolute continuity
  requirement allows one to remove from $X$ any subset with zero measure.} in Figure \ref{fig:ch5:ot:OTplans}.

The transport plans of the two other examples (C) and (D) have no
associated transport map, because they split Dirac masses.
The transport plan associated with Figure \ref{fig:monge_pb}
has the same nature (but this time in $2D \times 2D = 4D$). It cannot be written
with the form $(Id \times T)\sharp \mu$ because it splits the mass concentrated in $L_1$
into $L_2$ and $L_3$. 

Now the theoretical questions are:
\begin{itemize}
\item when does an optimal transport plan exist ?
\item when does it admit an associated optimal transport map ?
\end{itemize}

A standard approach to tackle this type of
existence problem is to find a certain \emph{regularity} both in the functional
and in the space of the admissible transport plans, that is, proving that the
functional is sufficiently \entreguillemets{smooth} and finding a compact set
of admissible transport plans. Since the set of admissible transport plans
contains at least the product measure $\mu \otimes \nu$, it is non-empty, and
existence can be proven thanks to a topological argument that exploits the
regularity of the functional and the compactness of the set. Once the existence
of a transport plan is established, the other question concerns the existence of
an associated transport map. Unfortunately, problem \eqref{eqn:K} does not directly show the
structure required by this reasoning path. However, one can observe that \eqref{eqn:K} is a
\emph{linear} optimization problem subject to \emph{linear} constraints. This suggests
using certain tools, such as the dual formulation, that was also developed by Kantorovich.
With this dual formulation, it is possible to exhibit an interesting structure of the
problem, that can be used to answer both questions (existence of a transport plan,
and existence of an associated transport map).

\section{Kantorovich's dual problem}
\label{ch5:sec:ot:KD}

Kantorovich's duality applies to problem \eqref{eqn:K}, in its most
general form (with measures). To facilitate understanding, we will
consider instead a discrete version of problem \eqref{eqn:K}, where the involved
entities are vectors and matrices (instead of measures and operators).
This makes it easy to better discern the structure of \eqref{eqn:K}, that is
the linear nature of both the functional and the constraints. This
also makes it easier for the reader to understand how to construct
the dual by manipulating simpler objects (matrices and vectors), however
the structure of the reasoning is the same in the general case.

\subsection{The discrete Kantorovich problem}
In Figure \ref{fig:DK}, we show a discrete version of the 1D transport between
two segments of Figure \ref{fig:ch5:ot:OTplans}.
The measures $\mu$ and $\nu$
are replaced with vectors $U = (u_i)_{i=1\ldots m}$ and $V = (v_j)_{i=1\ldots n}$.
The transport plan $\gamma$ becomes a set of coefficients $\gamma_{ij}$. Each 
coefficient $\gamma_{ij}$ indicates the quantity of matter  that will be transported
from $u_i$ to $v_j$.
The discrete Kantorovich problem can be written as follows:
\begin{equation}
\quad \quad
   \min\limits_\gamma <C,\gamma> \mbox{ subject to }
      \left\{
       \begin{array}{l}
         {\bf P_x} \gamma = U \\
         {\bf P_y} \gamma = V \\
         \gamma_{i,j} \ge 0 \quad \forall i,j
       \end{array}  
      \right.
\label{chap5:eqn:ot:primal}      
\end{equation}
where $\gamma$ is the vector of $\mathbb{R}^{m \times n}$ with all
coefficients $\gamma_{ij}$ (that is, the matrix $\gamma_{ij}$ \entreguillemets{unrolled}
into a vector), and $C$ the vector of $\mathbb{R}^{m \times n}$ with the coefficients $c_{ij}$
indicating the transport cost between point $i$ and point $j$ (for instance, the
Euclidean cost). The objective function is simply the dot product, denoted by $<C,\gamma>$,
of the cost vector $C$ and the vector $\gamma$. The objective function is \textbf{linear in $\gamma$}.
The constraints on the marginals \eqref{eqn:ch5:ot:marginales} impose in this discrete version
that the sums of the $\gamma_{ij}$ coefficients over the columns correspond to the $u_i$ coefficients (Figure \ref{fig:DK}-B)
and the sums over the rows correspond to the $v_j$ coefficients (Figure \ref{fig:DK}-C).
Intuitively, everything that leaves point $i$ should correspond to $u_i$, that is the quantity
of matter initially present in $i$ in the source landscape, and everything that arrives at a point 
$j$ should correspond to $v_j$, that is the quantity of matter desired at $j$ in
the target landscape. As one can easily notice, in this form, both constraints are
\textbf{linear in $\gamma$}. They can be written with two matrices ${\bf P_x}$ and ${\bf P_y}$,
of dimensions $m \times mn$ and $n \times mn$ respectively. 

\subsection{Constructing the Kantorovich dual in the discrete setting}

We introduce, arbitrarily for now, the following function $L$ defined by:
$$
    L(\varphi, \psi) = <C,\gamma> - <\varphi, {\bf P_x} \gamma - U> - <\psi, {\bf P_y} \gamma - V>
$$
that takes as arguments two vectors, $\varphi$ in ${\mathbb R}^m$ and $\psi$ in
${\mathbb R}^n$. The function $L$ is constructed from the objective function $<C,\gamma>$ from which
we subtracted the dot products of $\varphi$ and $\psi$ with the vectors that correspond to the
degree of violation of the constraints. One can observe that:
$$
\begin{array}{lcl}
  \sup\limits_{\varphi,\psi}[ L(\varphi,\psi) ] & = & <C,\gamma> \mbox{ if } {\bf P_x} \gamma = U \mbox{ and } {\bf P_y} \gamma = V \\
                                            & = & +\infty \mbox{ otherwise.}
\end{array}
$$
Indeed, if for instance a component $i$ of ${\bf P_x}\gamma$ is non-zero, one can make $L$ arbitrarily large
by suitably choosing the associated coefficient $\varphi_i$.

Now we consider:
$$
  \inf\limits_{\gamma \ge 0}\left[ \sup\limits_{\varphi, \psi} [L(\varphi, \psi)] \right] =
  \inf\limits_{\begin{tiny}\begin{array}{l} \gamma \ge 0 \\ {\bf P_x} \gamma = U \\ {\bf P_y} \gamma = V \end{array}\end{tiny}} \left[ <C, \gamma>  \right].
$$
There is equality, because to minimize $\sup[ L(\varphi, \psi) ]$, $\gamma$ has no other choice than
satisfying the constraints (see the previous observation). Thus, we obtain a new expression (left-hand side)
of the discrete Kantorovich problem (right-hand side). We now further examine it, and replace $L$ by its expression:
\begin{eqnarray}
\vspace{-5mm}\label{chap5:eqn:ot:dual1}
&  &\hspace{-7mm} \inf\limits_{\gamma \ge 0} \left[ \sup\limits_{\varphi, \psi} 
    \left(\begin{array}{ll}
      <C,\gamma> & - <\varphi, {\bf P_x} \gamma - U> \\
                 & - <\psi, {\bf P_y} \gamma - V> 
    \end{array}\right)
  \right]  \\[3mm]
\label{chap5:eqn:ot:dual2} 
& = & \sup\limits_{\varphi, \psi} \left[ \inf\limits_{\gamma \ge 0}
  \left(\begin{array}{ll}
  <C,\gamma> & - <\varphi, {\bf P_x} \gamma - U> \\
             & - <\psi, {\bf P_y} \gamma - V> 
  \end{array}\right)
 \right] \\[3mm]
\label{chap5:eqn:ot:dual3} 
& = & \sup\limits_{\varphi, \psi} \left[ \inf\limits_{\gamma \ge 0}
  \left(
    \begin{array}{l}
      <\gamma, C-{\bf P_x}^t \varphi - {\bf P_y}^t \psi> + \\[1mm]
      <\varphi, U> + <\psi, V>
    \end{array} 
  \right)
  \right] \\[3mm]
\label{chap5:eqn:ot:dual4} 
& = & \sup\limits_{\begin{tiny} \begin{array}{c} \varphi, \psi \\ {\bf P_x}^t \varphi + {\bf P_y}^t \psi \le C\end{array}\end{tiny}} \left[ <\varphi,U> + <\psi,V> \right]. \quad \quad 
\end{eqnarray}

The first step \eqref{chap5:eqn:ot:dual2} consists in exchanging the \entreguillemets{$\inf$} and \entreguillemets{$\sup$}.
Then we rearrange the terms \eqref{chap5:eqn:ot:dual3}.
By reinterpreting this equation as a constrained optimization problem (similarly to what we did in the previous paragraph), we finally
obtain the constrained optimization problem in \eqref{chap5:eqn:ot:dual4}. In the constraint ${\bf P_x}^t \varphi + {\bf P_y}^t \psi \le C$, the inequality
is to be considered componentwise. Finally, the problem \eqref{chap5:eqn:ot:dual4} can be rewritten as:

\begin{equation}
\begin{array}{l}
   \sup\limits_{\varphi, \psi} \left[ <\varphi, U> + <\psi, V>  \right] \\[2mm]
      \mbox{ subject to } \varphi_i + \psi_j \le c_{ij}, \quad \forall i,j.
\end{array}   
\end{equation}

As compared to the primal problem \eqref{chap5:eqn:ot:primal} that depends on $m \times n$ variables (all the coefficients
$\gamma_{ij}$ of the optimal transport plan for all couples of points $(i,j)$), this dual problem depends on
$m + n$ variables (the components $\varphi_i$ and $\psi_j$ attached to the source points and target points).
We will see later how to further reduce the number of variables, but before then,
we go back to the general continuous setting (that is, with functions, measures and operators).
    
\subsection{The Kantorovich dual in the continuous setting}
The same reasoning path can be applied to the continuous Kantorovich problem
\eqref{eqn:K}, leading to the following problem (DK):

\begin{equation}
   \begin{array}{l}
   (DK) \quad \quad \sup\limits_{\varphi,\psi} \left[ \int\limits_X \varphi d\mu + \int\limits_Y \psi d\nu \right ] \\[4mm]
   \mbox{ subject to: } \\
     \varphi(x) + \psi(y) \le c(x,y) \quad \forall (x,y) \in X \times Y,
  \end{array}
  \label{eqn:DK}  
\end{equation} where $\varphi$ and $\psi$ are now functions defined on $X$ and $Y$\footnote{The functions
$\varphi$ and $\psi$ need to be taken in $L^1(\mu)$ and $L^1(\nu)$. The proof of the equivalence with problem \eqref{eqn:K}
requires more precautions than in the discrete case, in particular step \eqref{chap5:eqn:ot:dual2}
(exchanging $\mbox{sup}$ and $\mbox{inf}$), that uses a result of convex analysis (due to Rockafellar),
see \cite{OTON} chapter 5.}. \\

The classical image that gives an intuitive meaning to this dual problem is to consider that instead of
transporting earth by ourselves, we are now hiring a company that will do the work on our behalf. The
company has a special way of determining the price: the function $\varphi(x)$ corresponds to what they charge for
loading earth at $x$, and $\psi(y)$ corresponds to what they charge for unloading earth at $y$.
The company aims at maximizing its profit (this is why the dual problem is a \entreguillemets{$\mbox{sup}$} rather
than an \entreguillemets{$\mbox{inf}$)},
but it cannot charge more than what it would cost us if we were doing the work by ourselves (hence the constraint). \\

The existence of solutions for (DK) remains difficult to study, because the set of functions $\varphi,\psi$ that
satisfy the constraint is not compact. However, it is possible to reveal more structure of the problem,
by introducing the notion of \emph{c-transform}, that makes it possible to exhibit a set
of admissible functions with sufficient regularity:

\begin{definition}
\begin{itemize}

\item For $f$  any function on $Y$ with values in $\mathbb{R} \cup \{ - \infty \}$ and not identically $-\infty$, we define its $c$-transform by 
\[
f^c(x)=\inf\limits_{y \in Y}\left[ c(x,y)-f(y) \right], \quad x \in X.
\]

  \item If a function $\varphi$ is such that there exists a function $f$ such that $\varphi = f^c$, 
     then $\varphi$ is said to be \emph{$c$-concave};
   \item ${\bf \Psi}_c(X)$ denotes the set of c-concave functions on $X$.
\end{itemize}
\end{definition}


We now show two properties of (DK) that will allow us to restrict the problem to
the class of $c$-concave functions to search for $\varphi$ and $\psi$:

\begin{observation}
  If the pair $(\varphi,\psi)$ is admissible for (DK), then the pair $(\psi^c,\psi)$ is admissible as well.
\end{observation}
\begin{proof}
$$
\begin{array}{l}
    \left\{
        \begin{array}{l}
           \forall(x,y) \in X \times Y, \varphi(x) + \psi(y) \le c(x,y) \\
           \psi^c(x) = \inf\limits_{y \in Y} \left[ c(x,y) - \psi(y) \right]
        \end{array}
    \right. \\[3mm]
    \begin{array}{lcl}
        \psi^c(x) + \psi(y) & = & \inf\limits_{y^\prime \in Y}\left[  c(x,y^\prime) - \psi(y^\prime)  \right] + \psi(y)\\
                            & \le & (c(x,y) - \psi(y))+\psi(y) \\
                            & \le & c(x,y).
    \end{array}
\end{array}
$$
\end{proof}

\begin{observation}
   If the pair $(\varphi,\psi)$ is admissible for (DK), then one obtains a \emph{better pair} by replacing $\varphi$ with $\psi^c$:
\end{observation}
\begin{proof}
$$
\left.
   \begin{array}{lcl}
   & \psi^c(x)    =  & \inf\limits_{y \in Y} \left[ c(x,y) - \psi(y) \right] \\
       \forall y \in Y, \quad &\varphi(x)  \le & c(x,y) - \psi(y)
   \end{array}
\right\} 
   \Rightarrow
   \psi^c(x) \le \varphi(x).
$$
\end{proof}

In terms of the previous intuitive image, this means that by replacing $\varphi$ with $\psi^c$, the company can charge more while
the price remains acceptable for the client (that is, the constraint is satisfied). Thus, we have:
$$
\begin{array}{ll}
   \inf(K)  & = \sup\limits_{\varphi \in {\bf \Psi}_c(X)} \quad \int\limits_X \varphi \ d\mu + \int\limits_Y \varphi^c \ d\nu \\[4mm]
            & = \sup\limits_{\psi \in {\bf \Psi}_c(Y)} \quad  \int\limits_X \psi^c \ d\mu+ \int\limits_Y \psi \ d\nu 
\end{array}   
$$   

We do not give here the detailed proof for existence. The reader is 
referred to \cite{OTON}, Chapter 4. The idea is that we are now in
a much better situation, since the set of admissible functions
${\bf \Psi}_c(X)$ is compact\footnote{Provided that the value of
$\psi$ is fixed at a point of $Y$ in order to suppress invariance
with respect to adding a constant to $\psi$.}. \\

The optimal value gives an interesting information, that is the minimum cost of transforming $\mu$ into $\nu$.
This can be used to define a distance between distributions, and also gives a way to compare different distributions,
which is of practical interest for some applications. 

\section{From the Kantorovich dual to the optimal transport map}

\subsection{The $c$-superdifferential}

Suppose now that in addition to the optimal cost you want to know the associated way to transform $\mu$ into $\nu$,
in other words, when it exists, the map $T$ from $X$ to $Y$ which associated transport plan $(Id \times T)\sharp\mu$ minimizes
the functional of the Monge problem. A result characterizes
the support of $\gamma$, that is the subset $\partial^{c} \varphi \subset X \times Y$ of the pairs of points $(x,y)$ connected by the transport plan:


\begin{theorem}
Let $\varphi$ a $c$-concave function. For all $(x,y) \in \partial^{c}\varphi$, we have
$$
\nabla \varphi(x) - \nabla_x c(x,y) = 0,
$$
where $ \partial^{c} \varphi = \{ (x,y) \ |  \ \varphi(z)\leq  \varphi(x)  +(c(z,y)-c(x,y)), \forall z \in X \}\footnote{By definition of the $c$-transform,  if $(x,y)\in \partial^{c} \varphi$, then $\varphi^c(y) = c(x,y)-\varphi(x)$. Then, the $c$-superdifferential can be characterized by the set of all points $(x,y)\in X\times Y$ such that
$ \varphi(x)+\varphi^c(y) =c(x,y)$.}$ denotes the so-called \emph{$c$-superdifferential} of $\varphi$.
\label{thm:MongeSol}
\end{theorem}
\begin{proof}
See \cite{OTON} chapters 9 and 10.
\end{proof}

In order to give an idea of the relation between the $c$-superdifferential and the associated transport map $T$, we present below a
heuristic argument:
consider a point $(x,y)$ in the $c$-superdifferential $\partial^{c} \varphi$, then for all $ z \in X$ we have
\begin{equation}
c(x,y)-\varphi(x) \leq c(z,y)-\varphi(z) .
\label{eqn:ssdiff1}
\end{equation}

Now, by using \eqref{eqn:ssdiff1}, we can compute the derivative at $x$ with respect to an arbitrary direction $w$
\[
\lim_{t\to 0^+}\frac{\varphi(x + t w)-\varphi(x)}{t}\leq \lim_{t\to 0^+}\frac{c(x + t w,y)-c(x,y)}{t}
\]
and we obtain $\nabla \varphi(x) \cdot w \le \nabla_x c(x,y) \cdot w$. We can do the same derivation along direction $-w$
instead of $w$, and then we get $\nabla \varphi(x) \cdot w = \nabla_x c(x,y) \cdot w$, $\forall w\in X$. 
%


In the particular case of the $L_2$ cost, that is with $c(x,y) = 1/2 \| x - y \|^2$,
this relation becomes $\forall
(x,y) \in \partial^{c} \varphi, \nabla \varphi(x) + y - x = 0$, thus, when
the optimal transport map $T$ exists, it is given by
\[
T(x) = x - \nabla \varphi(x) = \nabla (\|x\|^2/2 - \varphi(x)).
\] 
Not
only this gives an expression of $T$ in function of $\varphi$, which is
of high interest to us if we want to compute the transport explicitly.
In addition, this makes it possible to characterize $T$ as the gradient of a \emph{convex} function (see also
Brenier's polar factorization theorem \cite{BrenierPFMR91}). This
convexity property is interesting, because it means that
two \entreguillemets{transported particles} $x_1 \mapsto T(x_1)$ et
$x_2 \mapsto T(x_2)$ will never collide. We now see how to prove
these two assertions ($T$ gradient of a convex function
and absence of collision) in the case of the $L_2$ transport (with $(c(x,y) = 1/2 \| x - y \|^2$).



\begin{figure}
      \centerline{
         \includegraphics[width=50mm]{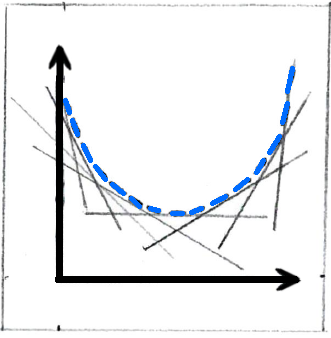}
      }
      \caption{
         The upper envelope of a family of affine functions is a convex function.
      }
      \label{fig:convex}
\end{figure}

\begin{figure*}
      \centerline{
         \includegraphics[width=\textwidth]{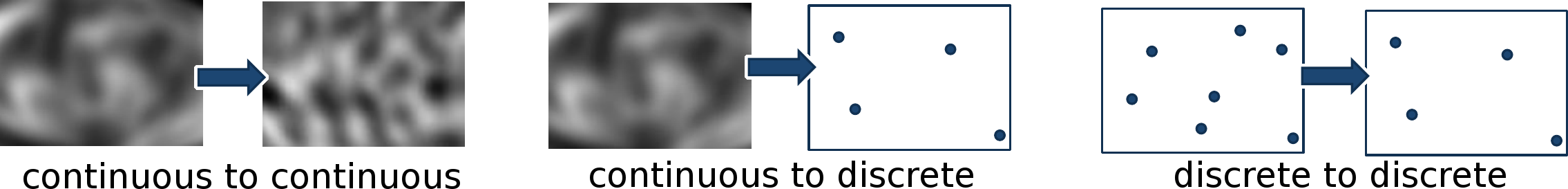}
      }
      \caption{
        Different types of transport, with continuous and discrete measures $\mu$ and $\nu$.
      }
      \label{fig:transport_types}
\end{figure*}

\begin{observation}
If $c(x,y)$ = $1/2 \| x - y \|^2$ and $\varphi \in {\bf \Psi}_c(X)$, then $\bar{\varphi}: x \mapsto \bar{\varphi}(x) = \| x \|^2/2 - \varphi(x)$
is a convex function (it is an equivalence if $X = Y = \mathbb{R}^d$, see \cite{Santambrogio}).
\label{obs:PhiConvexity}
\end{observation}
\begin{proof}
\[
\begin{split}
   \varphi(x)  = & \psi^c(x) \\
            = & \inf\limits_y \left[\frac{\|x - y\|^2}{2} - \psi(y)\right] \\
            = & \inf\limits_y \left[\frac{\|x\|^2}{2} - x \cdot y + \frac{\|y\|^2}{2} - \psi(y)\right].
\end{split}
\]
Then,
\[
\begin{split}
-\bar{\varphi}(x)  = & \varphi(x) - \frac{\|x\|^2}{2} = \inf\limits_y \left[ -x \cdot y + \left( \frac{\|y\|^2}{2} - \psi(y)  \right) \right].
\end{split}
\]
Or equivalently,
\[
\begin{split}
\bar{\varphi}(x) = & \sup\limits_y \left[ x \cdot y - \left( \frac{\|y\|^2}{2} - \psi(y) \right) \right].
\end{split}
\]
The function $ x \mapsto x \cdot y - \left( \frac{\|y\|^2}{2} - \psi(y) \right)$ is affine in $x$, therefore the graph of
$\bar{\varphi}$ is the upper envelope of a family of hyperplanes, therefore $\bar{\varphi}$ is a convex function (see Figure \ref{fig:convex}).
\end{proof}


\begin{observation}
We now consider the trajectories of two particles parameterized by $t \in [0,1]$, $t \mapsto (1-t)x_1 + t T(x_1)$ and
$t \mapsto (1-t)x_2 + t T(x_2)$. If $x_1 \neq x_2$ and $0 < t < 1$ then there is no collision between the two
particles.
\end{observation}
\begin{proof}
  By contradiction, suppose there is a collision, that is there exists 
  $t \in (0,1)$ and $x_1 \neq x_2$ such that
\[  
 (1-t)x_1 + tT(x_1)  = (1-t)x_2 + tT(x_2).
 \]
Since $T=\nabla \bar{\varphi}$, we can rewrite the last equality as
\[ 
 (1-t)(x_1 - x_2)  + t (\nabla \bar{\varphi}(x_1) - \nabla \bar{\varphi}(x_2))=0.
\]
Therefore, 
 \[
   (1-t)\| x_1 - x_2 \|^2 + t(\nabla \bar{\varphi}(x_1) - \nabla \bar{\varphi}(x_2))\cdot (x_1 - x_2) = 0.
\]
  The last step leads to a contradiction, between the left-hand side is the sum of
two strictly positive numbers (recalling the definition of the convexity of   
$\bar{\varphi}$: 
   $\forall x_1 \neq x_2, (x_1-x_2) \cdot (\nabla \bar{\varphi}(x_1) - \nabla \bar{\varphi}(x_2)) > 0$
)\footnote{Note that even if there is no collision, the trajectories can cross, that is
  $(1-t)x_1 + tT(x_1)  = (1-t')x_2 + t'T(x_2)$ for some $t\neq t'$ (see example in \cite{OTON}).
  If the cost is the Euclidean distance (instead of squared Euclidean distance),
  the non-intersection property is stronger and trajectories cannot cross. This comes at the expense of losing the
  uniqueness of the optimal transport plan\cite{OTON}.} 
\end{proof}

\section{Continuous, discrete and semi-discrete transport}

The properties that we have presented in the previous sections are true for any couple of
source and target measures $\mu$ and $\nu$, that can derive from continuous functions or
that can be discrete empirical measures (sum of Dirac masses). Figure \ref{fig:transport_types} presents three
configurations that are interesting to study. These configurations have specific properties, that lead to
different algorithms for computing the transport. We give here some indications and references concerning
the \emph{continuous $\rightarrow$ continuous} and \emph{discrete $\rightarrow$ discrete} cases. Then we will
develop the \emph{continuous $\rightarrow$ discrete} case with more details in the next section. 

\subsection{The continuous $\rightarrow$ continuous case and Monge-Amp\`ere equation}
We recall that when the optimal transport map exists, in the case of the $L_2$ cost
(that is, $c(x,y) = 1/2\| x - y \|^2$), it can be deduced from the function $\varphi$ by using
the relation $T(x) = \nabla \bar{\varphi} = x - \nabla \varphi$. The change of variable formula for integration
over a subset $B$ of $X$ can be written as:
\begin{equation}
   \forall B \subset X, \int_B 1 d\mu = \mu(B) = \nu(T(B)) = \int_B \left| \det J_T(x) \right|  d\mu
   \label{eqn:changevar}
\end{equation}
where $J_T$ denotes the Jacobian matrix of $T$ and $\det$ denotes the determinant.

If $\mu$ and $\nu$ have densities $u$ and $v$ respectively, that is
$\forall B, \mu(B) = \int_B u(x)dx$ and $\nu(B) = \int_B v(x)dx$, then one can (formally) consider
\eqref{eqn:changevar} pointwise in $X$:
\begin{equation}
    \forall x \in X, \ u(x) = \left|\det J_T(x)\right| v(T(x)).
   \label{eqn:pointwise}
\end{equation}
By injecting $T=\nabla\bar{\varphi}$ and $J_T = H \bar{\varphi}$ into \eqref{eqn:pointwise}, one obtains:
\begin{equation}
\forall x \in X, \ u(x) = \left| \det H \bar{\varphi}(x) \right| v (\nabla \bar{\varphi}(x)),
\label{eqn:MAE}
\end{equation}
where $H\bar{\varphi}$ denotes the Hessian matrix of $\bar{\varphi}$. Equation
\eqref{eqn:MAE} is known ad the \emph{Monge-Amp\`ere equation}. It is a highly non-linear
equation, and its solutions when they exist often present singularities\footnote{
This is similar to the eikonal equation, which solution corresponds to the distance field,
that has a singularity on the medial axis.}. Note that the derivation above is purely formal,
and that studying the solutions of the Monge-Amp\`ere equation require using more sophisticated
tools. In particular, it is possible to define several types of weak solutions (viscosity solutions,
solution in the sense of Brenier, solutions in the sense of Alexandrov \ldots). Several algorithms
to compute numerical solutions of the Monge-Amp\`ere equations were proposed. As such, see for
instance the Benamou-Brenier algorithm \cite{ACFM:BB:2000}, that uses a dynamic formulation
inspired by fluid dynamics (incompressible Euler equation with specific boundary conditions).
See also \cite{papadakis:hal-00816211}. 

\subsection{The discrete $\rightarrow$ discrete case}
If $\mu$ is the sum of $m$ Dirac masses and $\nu$
the sum of $n$ Dirac masses, then the problem boils down to finding the $m \times n$ coefficients $\gamma_{ij}$
that give for each pair of points $i$ of the source space and $j$ of the target space
the quantity of matter transported from $i$ to $j$. This corresponds to the transport plan
in the discrete Kantorovich problem that we have seen previously \S\ref{ch5:sec:ot:KD}.
This type of problem (referred to as an \emph{assignment problem}) can be solved by
different methods of linear programming \cite{AP:BDM:2009}. These method can be dramatically
accelerated by adding a regularization term, that can be interpreted as the entropy of
the transport plan \cite{ChristianLeonardSchroedinger}. This regularized version of 
optimal transport can be solved by highly efficient numerical algorithms \cite{DBLP:conf/nips/Cuturi13}.

\subsection{The continuous $\rightarrow$ discrete case}
This configuration, also called \emph{semi-discrete}, corresponds to a continuous function transported to a sum of Dirac masses (see the examples of c-concave functions
in \cite{gangbo1996}). This correspond to our example with factories that consume a resource, in \S\ref{sect:monge_measure}.
Semi-discrete transport has interesting connections with some notions of
computational geometry and some specific sets of convex polyhedra that were studied by Alexandrov
\cite{Alexandrov48} and later by Aurenhammer, Hoffman and Aranov \cite{DBLP:conf/compgeom/AurenhammerHA92}.
The next section is devoted to this configuration.

\section{Semi-discrete transport}
\label{chap5:sec:ot:semidiscreteOT}

\begin{figure}
      \centerline{ \includegraphics[width=50mm]{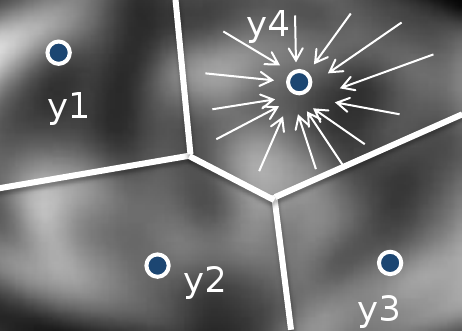}}

\caption{
  Semi-discrete transport: gray levels symbolize the quantity of a resource that will
  be transported to 4 factories. Each factory will be allocated a part of the terrain
  in function of the quantity of resource that it should collect.
}
\label{fig:semi_discrete}
\end{figure}

We now suppose that the source measure $\mu$ is continuous, and that the target measure
$\nu$ is a sum of Dirac masses. A practical example of this type of configuration corresponds
to a resource which available quantity is represented by a function $u$. The resource
is collected by a set of $n$ factories, as shown in Figure \ref{fig:semi_discrete}.
Each factory is supposed to collect a certain prescribed quantity of resource $\nu_j$. Clearly, the sum
of all prescriptions corresponds to the total quantity of available resource ($\sum_{j=1}^n \nu_j= \int_X u(x) dx$). \\

\subsection{Semi-discrete Monge problem}
In this specific configuration, the Monge problem becomes:

$$
  \begin{array}{l}
    \inf\limits_{T: X\rightarrow Y} \int\limits_X c(x,T(x)) u(x)dx,
     \mbox{ subject to} \int\limits_{T^{-1}(y_j)} u(x)dx = \nu_j,  \forall j.
    \end{array}
   \label{eqn:sdmonge} 
$$

 A transport map $T$ associates with each point $x$ of $X$ one of the points $y_j$. Thus, it is possible
to partition $X$, by associating to each $y_j$ the region $T^{-1}(y_j)$ that contains all the points transported
towards $y_j$ by $T$. The constraint imposes that the quantity of collected resource over each region $T^{-1}(y_j)$
corresponds to the prescribed quantity $\nu_j$. \\

Let us now examine the form of the dual Kantorovich problem.
In terms of measure, the source measure $\mu$ has a 
density $u$, and the target measure $\nu$ is a sum
of Dirac masses $\nu = \sum_{j=1}^n \nu_j \delta_{y_j}$, supported by the
set of points $Y = \left\{ y_j \right\}$. We recall that in its general
form, the dual Kantorovich problem is written as follows:
\begin{equation}
   \sup\limits_{\psi \in {\bf \Psi^c}(Y)} \left[ \int_X \psi^c(x)d\mu + \int_Y \psi(y)d\nu \right].
   \label{eqn:sdk1}   
\end{equation}

In our semi-discrete case, the functional becomes a function of $n$ variables, with the following form:
\begin{eqnarray}
   \label{eqn:sdk00}
     F(\psi) &=& \!F(\psi_1, \psi_2, \ldots \psi_n) \\
   \label{eqn:sdk2}   
    &=&\! \int\limits_X \! \psi^c(x) u(x) dx + \sum_{j=1}^n \psi_j \nu_j  \\
   \label{eqn:sdk3}
    &=&\! \int\limits_X \! \inf\limits_{y_j \in Y} \left[c(x,y_j) - \psi_j \right] u(x) dx + \sum_{j=1}^n \psi_j \nu_j \\
   \label{eqn:sdk4}
   &=& \! \sum\limits_{j=1}^n \int\limits_{\Lag_\psi^c(y_j)} \!\left( c(x, y_j) - \psi_j \right) u(x) dx + \sum_{j=1}^n \psi_j \nu_j. 
\end{eqnarray}

The first step \eqref{eqn:sdk2} takes into account the nature of the measures $\mu$ and $\nu$. 
In particular, one can notice that the measure $\nu$ is completely defined by the scalars $\nu_j$ associated
with the points $y_j$, and the function $\psi$ is defined by the scalars $\psi_j$ that correspond to its value
at each point $y_j$. The integral $\int_Y \psi(y)d\nu$ becomes the dot product 
$\sum_j \psi_j \nu_j$. Thus, the functional that corresponds to the dual Kantorovich problem becomes a function $F$
that depends on $n$ variables (the $\psi_j$). Let us now replace the c-conjugate $\psi^c$ with its expression, which gives
\eqref{eqn:sdk3}. The integral in the left term can be reorganized, by grouping 
the points of $X$ for which the same point $y_j$ minimizes $c(x,y_j) - \psi_j$, which gives \eqref{eqn:sdk4},
where the \emph{Laguerre cell} $\Lag_\psi^c(y_j)$ is defined by:
$$
    \Lag_\psi^c(y_j) = \left\{ x \in X \quad | \quad c(x,y_j) - \psi_j \le c(x,y_k) - \psi_k, \ \forall k\neq j \right\}.
$$

The Laguerre diagram, formed by the union of the Laguerre cells, is a classical structure in computational geometry.
In the case of the $L_2$ cost $c(x,y) = 1/2\| x - y \|^2$, it corresponds to the \emph{power diagram},
that was studied by Aurenhammer at the end of the 80's \cite{DBLP:journals/siamcomp/Aurenhammer87}. One of its
particularities is that the boundaries of the cells are rectilinear, making it reasonably easy to design computational
algorithms to construct them.


\begin{figure*}
      \centerline{
        \includegraphics[width=50mm]{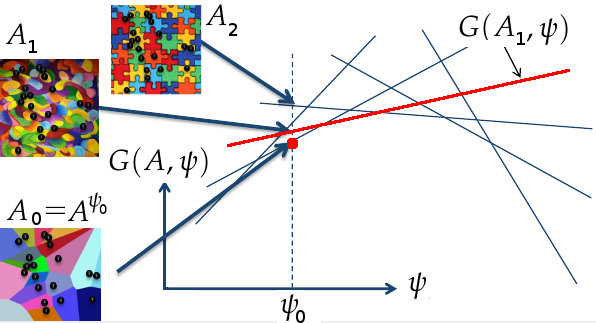}
         \hspace{5mm}
         \includegraphics[width=50mm]{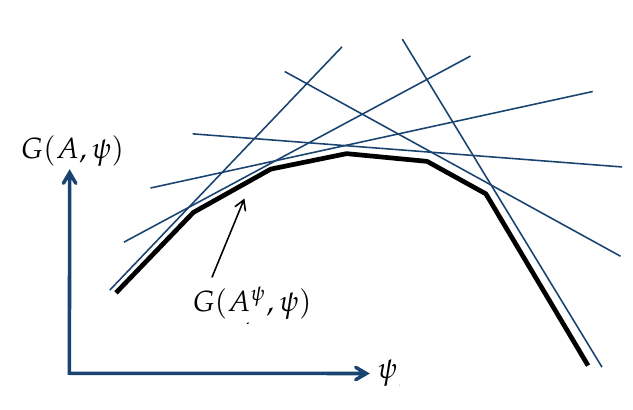}         
      }
      \caption{
        The objective function of the dual Kantorovich problem is concave, because its graph
        is the lower envelope of a family of affine functions.
      }
      \label{fig:concave}
\end{figure*}

\subsection{Concavity of $F$}
The objective function $F$ is a particular case of the Kantorovich dual,
and naturally inherits its properties, such as its concavity (that we did
not discuss yet). This property is interesting both from a theoretical
point of view, to study the existence and uniqueness of solutions, and from
a practical point of view, to design efficient numerical solution mechanisms.
In the semi-discrete case, the concavity of $F$ is easier to prove than in
the general case. We summarize here the proof by 
Aurenhammer \emph{et. al} \cite{DBLP:conf/compgeom/AurenhammerHA92},
that leads to an efficient algorithm
\cite{DBLP:journals/cgf/Merigot11}, \cite{journals/M2AN/LevyNAL15}, \cite{DBLP:journals/corr/KitagawaMT16}.

\begin{theorem}
The objective function $F$ of the semi-discrete Kantorovich dual problem \eqref{eqn:sdk1} is concave.
\end{theorem}
\begin{proof}
Consider the function $G$ defined by:
\begin{equation}
G(A, \left[ \psi_1, \ldots \psi_n \right]) = \int\limits_X \left( c(x, y_{A(x)}) - \psi_{A(x)} \right) u(x) dx,
\end{equation}
and parameterized by an \emph{assignment} $A~: X \rightarrow [1 \ldots n]$, that is, a function that
associates with each point $x$ of $X$ the index $j$ of one of the points $y_j$.
If we denote $A^{-1}(j) = \left\{ x | A(x) = j \right\}$, then $G$ can be also written as:
\begin{equation}
\begin{array}{lcl}
G(A,\psi) & = & \sum\limits_j \int\limits_{A^{-1}(j)} \left(c(x,y_j) - \psi_j\right) u(x) dx \\[2mm]
          & = & \sum\limits_j \int\limits_{A^{-1}(j)} c(x,y_j) u(x) dx - \sum\limits_j \psi_j \int\limits_{A^{-1}(j)} u(x) dx.
\end{array}
\label{ch4:eqn:ot:Gaffine}
\end{equation}
The first term does not depend on the $\psi_j$, and the second one is a linear combination of the $\psi_j$ coefficients, thus,
for a given fixed assignment $A$, $\psi \mapsto G(A, \psi)$ is an affine function of $\psi$. Figure \ref{fig:concave} depicts
the appearance of the graph of $G$ for different assignment. The horizontal axis symbolizes the components of the
vector $\psi$ (of dimension $n$) and the vertical axis the value of $G(A,\psi)$. For a given assignment $A$, the graph
of $G$ is an hyperplane (symbolized here by a straight line). 

Among all the possible assignments $A$, we distinguish $A^\psi$ that associates with a point $x$ the index $j$ of the Laguerre
cell $x$ belongs to \footnote{$A$ is undefined on the set of Laguerre cell boundaries, this does not cause any difficulty
because this set has zero measure.}, that is:
$$
    A^\psi(x) = \arg \min\limits_j \left[ c(x,y_j) - \psi_j \right].
$$

For a fixed vector $\psi = \psi^0$, among all the possible assignments $A$,
the assignment $A^{\psi^0}$ minimizes the value $G(A, \psi^0)$,
because it minimizes the integrand pointwise (see Figure \ref{fig:concave} on the left.
Thus, since $G$ is affine with respect to $\psi$, the graph of the function $\psi \rightarrow
G(A^\psi,\psi)$ is the lower envelope of a family of hyperplanes (symbolized as straight lines
in Figure \ref{fig:concave} on the right), hence $\psi \rightarrow
G(A^\psi,\psi)$ is a concave function. Finally, the objective function $F$ of the
dual Kantorovich problem can be written as
$F(\psi) = G(A^\psi,\psi) + \sum_j \nu_j \psi_j$, that is, the sum of
a concave function and a linear function, hence it is also a concave function.
\end{proof}

\subsection{The semi-discrete optimal transport map}
Let us now examine the $c$-superdifferential $\partial^{c} \psi$, that is,
the set of points ($x \in X$, $y \in Y$) connected by the optimal
transport map. We recall that the $c$-superdifferential $\partial^{c} \psi$ can be defined alternatively as
$ \partial^{c} \psi = \left\{ (x,y_j) \ | \ \psi^c(x) + \psi_j = c(x,y_j) \right\}$. Consider
a point $x$ of $X$ that belongs to the Laguerre cell $\Lag_\psi^c(y_j)$. The 
$c$-superdifferential $[\partial^{c}\psi](x)$ at $x$ is defined by:
\begin{eqnarray}
\label{eqn:Tdisc1}
 [\partial^{c}\psi](x) &=& \left\{ y_k \ | \ \psi^c(x) + \psi_k = c(x,y_k) \right\}  \\
\label{eqn:Tdisc2}   
   & =& \left\{ y_k \ | \ \inf_{y_l} \left[ c(x,y_l) - \psi_l \right] + \psi_k = c(x,y_k) \right\}  \\
\label{eqn:Tdisc3}     
&  = &  \left\{ y_k \ | \ c(x,y_j) - \psi_j + \psi_k = c(x,y_k) \right\}  \\
\label{eqn:Tdisc4}     
 & =&   \left\{ y_k \ | \ c(x,y_j) - \psi_j  = c(x,y_k) - \psi_k \right\}  \\
\label{eqn:Tdisc5}          
 & =&   \left\{ y_j \right\}. 
\end{eqnarray}

In the first step \eqref{eqn:Tdisc2}, we replace $\psi^c$ by its definition, then we
use the fact that $x$ belongs to the Laguerre cell of $y_j$ \eqref{eqn:Tdisc3}, and finally,
the only point of $Y$ that satisfies \eqref{eqn:Tdisc4} is $y_j$ because we have supposed 
$x$ \emph{inside} the Laguerre cell of $y_j$. \\

To summarize, the optimal transport map $T$ moves each point $x$ to the point $y_j$
associated with the Laguerre cell $\Lag_\psi^c(y_j)$ that contains $x$. The vector $\psi_1, \ldots \psi_n$
is the unique vector that maximizes the discrete dual Kantorovich function $F$ such that $\psi$ is c-concave.
It is possible to show that $\psi$ is c-concave if and only if no Laguerre cell is empty of matter, that is the integral
of $u$ is non-zero on each Laguerre cell. Indeed, we have the following results:

\begin{theorem}\label{c-concava_always}
Let $Y=\{ y_1, \ldots, y_n \}$ be a set of $n$ points. Let $\psi$ be  any  function  defined on $Y$ such that $\Lag_\psi^c(y_i)$ are not empty sets. Then $\psi$ is a $c$-concave
function. 
\end{theorem}
\begin{proof}
By definition, for $i=1, 2, \ldots, k$
\[
\begin{split}
(\psi^c)^{c}(y_i)&=\inf \limits_{x \in X}\left[c(x,y_i) -\psi^c(x)  \right]\\
&\hspace{-1cm}=\inf \limits_{x \in X}\left[c(x,y_i) -\left( \inf\limits_{y_j \in Y}\left[ c(x,y_j)-\psi(y_j) \right]  \right)  \right]\\
&\hspace{-1cm}=\inf \limits_{x \in X}\left\{
\begin{array}{ll}
\psi(y_i), \quad \quad \quad \quad  \quad  \ \quad \quad \quad \quad \hbox{if $x \in \Lag_\psi^c(y_i)$ } \\
c(x,y_i) - (\overbrace{c(x,y_j)-\psi(y_j)}^{ \quad <  (c(x,y_i)-\psi(y_i))}), \ \hbox{if $x \in \Lag_\psi^c(y_j)$ $(j\neq i)$ } 
\end{array}
\right.\\
&\hspace{-1cm}=\psi(y_i).
\end{split}
\]
This allows to conclude that $\psi$ is a $c$-convex function (see \cite[Proposition 1.34]{Santambrogio}).
\end{proof}

Moreover, the converse of the theorem is also true:
\begin{theorem}\label{c-concava_always2}
Let $Y=\{ y_1, \ldots, y_n \}$ be a set of $n$ points. Let $\psi$ be  a $c$-concave
function  defined on $Y$. Then the sets  $\Lag_\psi^c(y_i)$ are not empty  for all $i=1, \ldots, n$,. \end{theorem}
\begin{proof}
Reasoning by contradiction, we suppose that $\psi$ is a $c$-concave function and there exist $i_0 \in \{ 1, \ldots, n \}$ such that $\Lag_\psi^c(y_{i_0})=\emptyset$. Then, from definition
\[
\begin{split}
\forall x\in X, \ \exists  j \in \{ 1, \ldots, n \} \ \hbox{with}\ j \neq i_0, \ \hbox{and}\ \epsilon_j>0\  \hbox{such that} \\
c(x,y_{i_0})-\psi(y_{i_0}) \geq c(x,y_j)-\psi(y_j)+\epsilon_j.
\end{split}
\]
We will write $x=x_j$. Thus,
\begin{equation}\label{abs}
\begin{split}
c(x_j,y_{i_0}) & -\left( \inf\limits_{y_j \in Y}\left[ c(x_j,y_j)-\psi(y_j) \right]  \right)\\
& \geq c(x_j,y_{i_0}) -\left( c(x_j,y_{i_0})-\psi(y_{i_0})-\epsilon_j \right)\\
&=\psi(y_{i_0})+\epsilon_j.
\end{split}
\end{equation}
Therefore,
\[
\begin{split}
(\psi^c)^{c}(y_{i_0})&=\inf \limits_{x \in X}\left[c(x,y_{i_0}) -\left( \inf\limits_{y_j \in Y}\left[ c(x,y_j)-\psi(y_j) \right]  \right)  \right]\\
&=\inf \limits_{j}\left[c(x_j,y_{i_0}) -\left( \inf\limits_{y_j \in Y}\left[ c(x_j,y_j)-\psi(y_j) \right]  \right)  \right]\\
&=\inf \limits_{j}\left[\psi(y_{i_0})+\epsilon_j \right]\\
&>\psi(y_{i_0}).
\end{split}
\]
This contradicts the fact that $\psi$ is a c-concave function, because  \cite[Proposition 1.34]{Santambrogio}.
\end{proof}


We now proceed to compute the first and second order derivatives of the objective function $F$. These derivatives
are useful in practice to design computational algorithms.


\subsection{First-order derivatives of the objective function}

Since it is concave,  $F$ admits a unique maximum $\psi^*$, characterized by $\nabla
F(\psi^*) = 0$ where $\nabla F$ denotes the gradient. Let us now examine the form of the gradient $\nabla F
= \nabla \left(G(A\psi, \psi) + \sum_{j=1}^n \psi_j \nu_j\right)$.
By replacing $G(A_\psi, \psi)$ with its
expression \eqref{ch4:eqn:ot:Gaffine}, one obtains:
\[
\begin{split} 
\frac{\partial G}{\partial \psi_j}(\psi)&=\lim_{t\to 0}\frac{G(\psi + t{e}_j)-G(\psi )}{t}\\
&\hspace{-1.5cm}=\lim_{t\to 0}\frac{1}{t} \left\{ \int_{X} \inf\left[ c(x,y_1)-\psi_1, \ldots, c(x,y_j)-\psi_j - t,\ldots, c(x,y_n)-\psi_n \right] \right.\\
&\hspace{0.3cm}\left.-\inf\limits_{i}\left[ c(x,y_i)-\psi_i \right] u(x)  dx\right\}.  
\end{split}
\]
Let $x \in \Lag_\psi^c(y_m)$, then for $t$ small enough we have
\[
\begin{split}
\inf & \left[ c(x,y_1)-\psi_1, \ldots, c(x,y_j)-\psi_j - t,\ldots, c(x,y_k)-\psi_n \right]\\
&=c(x,y_m)-\psi_m.
\end{split}
\]
Indeed, since $c(x,y_m)-\psi_m < c(x,y_i)-\psi_i$ for all $i\neq m$, in particular, if $m\neq j$: $c(x,y_m)-\psi_m < c(x,y_j)-\psi_j$.
Thus, for $t$ small enough 
\[
c(x,y_m)-\psi_m \leq c(x,y_j)-\psi_j -t.
\]
In the case that $m=j$,
\[
\begin{split}
\inf &\left[ c(x,y_1)-\psi_1, \ldots, c(x,y_j)-\psi_j - t,\ldots, c(x,y_k)-\psi_k \right]\\
&=c(x,y_j)-\psi_j - t.
\end{split}
\]
Consequently, we obtain that
\[
\begin{split}
\frac{\partial G}{\partial \psi_j}(\psi)= - \int_{\Lag_\psi^c(y_j)} u(x)  dx. 
\end{split}
\]
Recalling that the objective function $F$ is given by $F(\psi) = G(A_\psi, \psi) + \sum_j \nu_j \psi_j$, we finally obtain:
\begin{equation}
\frac{\partial F}{\partial \psi_j} = \nu_j - \int\limits_{\Lag_\psi^c(y_j)} u(x) dx.
\label{eqn:grad}
\end{equation}
Since the objective function $F$ is concave, it admits a unique maximum.
At this maximum, all the components of the gradient vanish, this implies
that the quantity of matter obtained at $y_j$, that is the integral of
$u$ on the Laguerre cell of $y_j$, corresponds to the prescribed quantity
of matter, that is $\nu_j$.

\subsection{Second order derivatives of the objective function}
The coefficients of the Hessian matrix $\partial^2 F / \partial \psi_i \partial \psi_j$
are slightly more difficult to compute, since here, we cannot invoke the envelope theorem.
We cannot avoid invoking Reynold's formula. We do not detail the computations for length
considerations, but give the final result.

In the particular case of the $L_2$ cost, that is $c(x,y) = 1/2\| x - y \|^2$, the second-order derivatives are given by:
\begin{equation}\label{eqn:hess}
\begin{split}
\frac{\partial^2 G}{\partial \psi_i\partial \psi_j}(\psi)&=\int\limits_{\Lag_\psi^c(y_i)\cap \Lag_\psi^c(y_j)} \frac{u(x)}{\| y_i-y_j\|} \ dS(x) \quad ( i\neq j),\\  
\frac{\partial^2 G}{\partial \psi_j^2}(\psi)&=-\sum\limits_{i \neq j}\frac{\partial^2 G}{\partial \psi_i\partial \psi_j}(\psi).
\end{split}
\end{equation}

\subsection{A computational algorithm for $L_2$ semi-discrete optimal transport}
\label{sect:algo}
With the definition of $F(\psi)$, the expression of its first order derivatives
(gradient $\nabla F$) and second order derivatives (Hessian matrix $\nabla^2 F = \left(\partial^2 F / \partial \psi_i \partial \psi_j\right)_{ij}$),
we are now equipped to design a numerical solution mechanism that computes semi-discrete optimal transport by
maximizing $F$, based on a particular version \cite{DBLP:journals/corr/KitagawaMT16} of Newton's optimization method \cite{Nocedal2006NO}:
$$
\begin{array}{ll}
   \mbox{\bf Input:}  & \mbox{a mesh that supports the source density } u \\
                        & \mbox{the points } (y_j)_{j=1}^n \\
                        & \mbox{the prescribed quantities } (\nu_j)_{j=1}^n \\[2mm]
   \mbox{\bf Output:} & \mbox{the (unique) Laguerre diagram }\Lag_\psi^c \mbox{ such that:}\\
                        & \quad \int\limits_{\Lag_\psi^c(y_j)} u(x)dx = \nu_j \quad \forall j\\[5mm]
                        \hline \\
   (1) & \psi \leftarrow [ 0 \ldots 0 ] \\
   (2) & \mbox{While convergence is not reached} \\
   (3) & \quad \mbox{Compute } \nabla F \mbox{ and } \nabla^2 F \\
   (4) & \quad \mbox{ Find } p \in {\mathbb R}^n \mbox{ such that } \nabla^2 F(\psi) p = -\nabla F(\psi) \\
   (5) & \quad \mbox{ Find the descent parameter } \alpha \\
   (6) & \quad \psi \leftarrow \psi + \alpha p \\
   (7) & \mbox{End while}
\end{array}
$$

The source measure is given by its density, that is a positive piecewise linear function $u$,
supported by a triangulated mesh (2D) or tetrahedral mesh (3D) of a domain $X$. The target measure
is discrete, and supported by the pointset $Y = (y_j)_{j=1}^n$. Each target point will receive the prescribed
quantity of matter $\nu_j$. Clearly, the prescriptions should be balanced with the available resource,
that is $\int_X u(x)dx = \sum_j \nu_j$. The algorithm computes for each point of the target measure
the subset of $X$ that is affected to it through the optimal transport, $T^{-1}(y_j) = \Lag_\psi^c(y_j)$,
that corresponds to the Laguerre cell of $y_j$. The Laguerre diagram is completely determined
by the vector $\psi$ that maximizes $F$. \\

Line (2) needs a criterion for convergence. The classical convergence
criterion for a Newton algorithm uses the norm of the gradient of
$F$.  In our case, the components of the gradient of $F$ have a
geometric meaning, since $\partial F / \partial \psi_j$ corresponds to
the difference between the prescribed quantity $\nu_j$ associated with
$j$ and the quantity of matter present in the Laguerre cell of $y_j$
given by $\int_{\Lag_\psi^c(y_j)}\!\! u(x)dx$. Thus, we can decide to stop
the algorithm as soon as the largest absolute value of a component
becomes smaller than a certain percentage of the smallest
prescription. Thus we consider that convergence is reached if $\max |
\nabla F_j | < \epsilon \min_j \nu_j$, for a user-defined $\epsilon$
(typically 1\% in the examples below). \\

Line (3) computes the coefficients of the gradient and the Hessian matrix of $F$, using \eqref{eqn:grad} and
\eqref{eqn:hess}. These computations involve integrals over the Laguerre cells and over their boundaries.
For the $L_2$ cost $c(x,y) = 1/2\| x - y \|^2$, the boundaries of the Laguerre cells are rectilinear, which
dramatically simplifies the computations of the Hessian coefficients \eqref{eqn:hess}. In addition,
it makes it possible to use efficient algorithms to compute the Laguerre diagram
\cite{DBLP:journals/cj/Bowyer81,journals/cj/Watson81}. Their implementation is available in several
programming libraries, such as GEOGRAM\footnote{\url{http://alice.loria.fr/software/geogram/doc/html/index.html}} et
CGAL\footnote{\url{http://www.cgal.org}}. Then one needs to compute the intersection between each Laguerre cell
and the mesh that supports the density $u$. This can be done with specialized algorithms \cite{journals/M2AN/LevyNAL15},
also available in GEOGRAM. \\

Line (4) finds the Newton step $p$ by solving a linear system. We use the Conjugate Gradient algorithm \cite{Stiefel1952}
with the Jacobi preconditioner. In our empirical experiments below, we stopped the conjugate iterations as soon
as $\| \nabla^2F p + \nabla F \| / \| \nabla F \| < 10^{-3}$. \\

Line (5) determines the descent parameter $\alpha$. A result due to M\'erigot and Kitagawa \cite{DBLP:journals/corr/KitagawaMT16}
ensures the convergence of the Newton algorithm if the measure of the smallest Laguerre cell remains larger than a certain
threshold (that is, half the smallest prescription $\nu_j$). There is also a condition on the norm of the gradient $\|\nabla F\|$
that we do not repeat here (the reader is referred to M\'erigot and Kitagawa's original article for more details). In our implementation,
starting with $\alpha = 1$, we iteratively divide $\alpha$ by two until
both conditions are satisfied.\\

Let us now make one step backwards and think about the original
definition of Monge's problem \eqref{eqn:monge}. We wish to stress
that the initial constraint (local mass conservation) that
characterizes transport maps was terribly difficult. It is remarkable
that after several rewrites (Kantorovich relaxation, duality,
c-convexity), the final problem becomes as simple as optimizing a
regular ($C^2$) concave function, for which computing the gradient and
Hessian is easy in the semi-discrete case and boils down to evaluating
volumes and areas in a Laguerre diagram. We wish also to stress that
the computational algorithm did not require to make any approximation
or discretization. The discrete, computer version is a particular
setting of the general theory, that fully describes not only transport
between smooth objects (functions), but also transport between less
regular objects, such as pointsets and triangulated meshes.  This is
made possible by the rich mathematical vocabulary (measures) on which
optimal transport theory acts. Thus, the computational algorithm is
an elegant, direct verbatim translation of the theory into a computer program. \\

We now show some computational results and list possible applications of this
algorithm.

\section{Results, examples and applications of $L_2$ semi-discrete transport}

\begin{figure*}
      \centerline{
         \includegraphics[width=\textwidth]{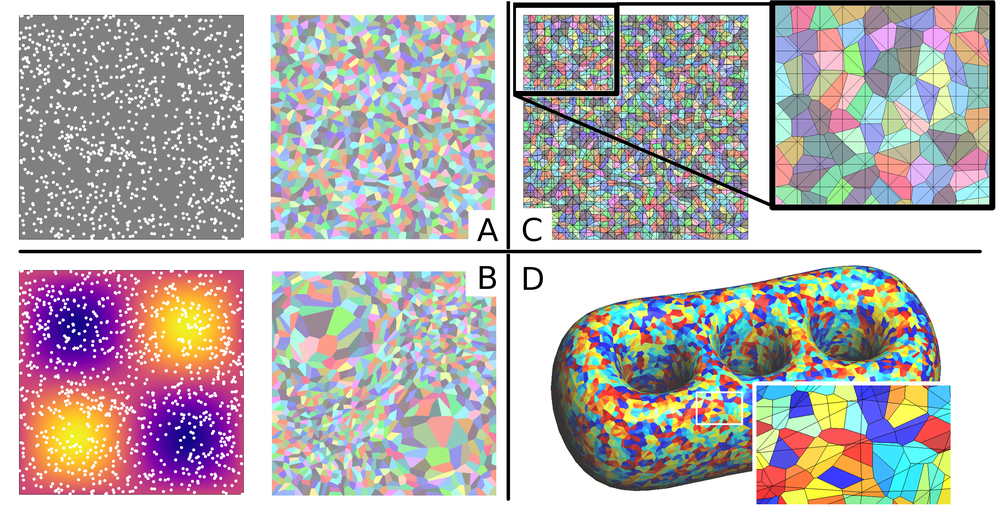}
      }
      \caption{A: transport between a uniform density and a random pointset;
        B: transport between a varying density and the same pointset;
        C: intersections between meshes used to compute the coefficients;
        D: transport between a measure supported by a surface and a 3D pointset.
      }
      \label{fig:exemples_2D}
\end{figure*}  

\begin{figure*}
      \centerline{
         \includegraphics[width=\textwidth]{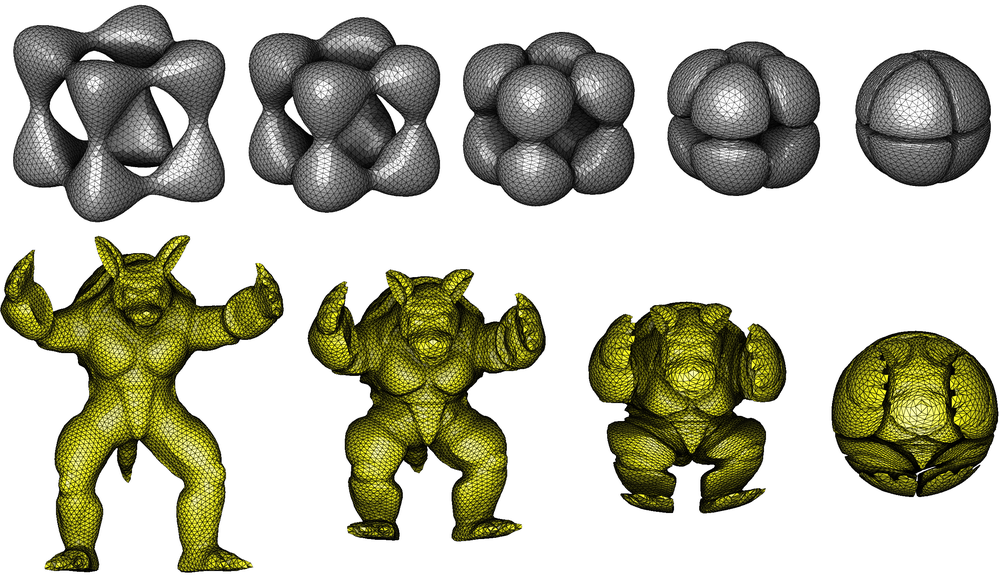}
      }
      \caption{Interpolation of 3D volumes using optimal transport.}
      \label{fig:exemple_3D}
\end{figure*}  

\begin{figure*}
      \centerline{
         \includegraphics[width=\textwidth]{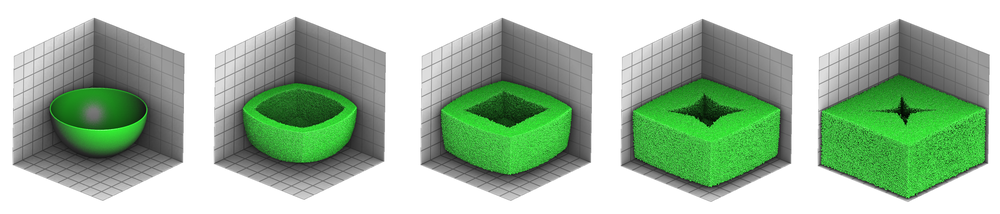}
      }
      \caption{Computing the transport between objects of different
      dimension, from a sphere to a cube. The sphere is sampled with
      10 million points. The cross-section reveals the formation of
      a singularity that has some similarities with the medial axis
      of the cube.}
      \label{fig:SphereCube}
\end{figure*}  

\begin{figure*}
      \centerline{
         \includegraphics[width=\textwidth]{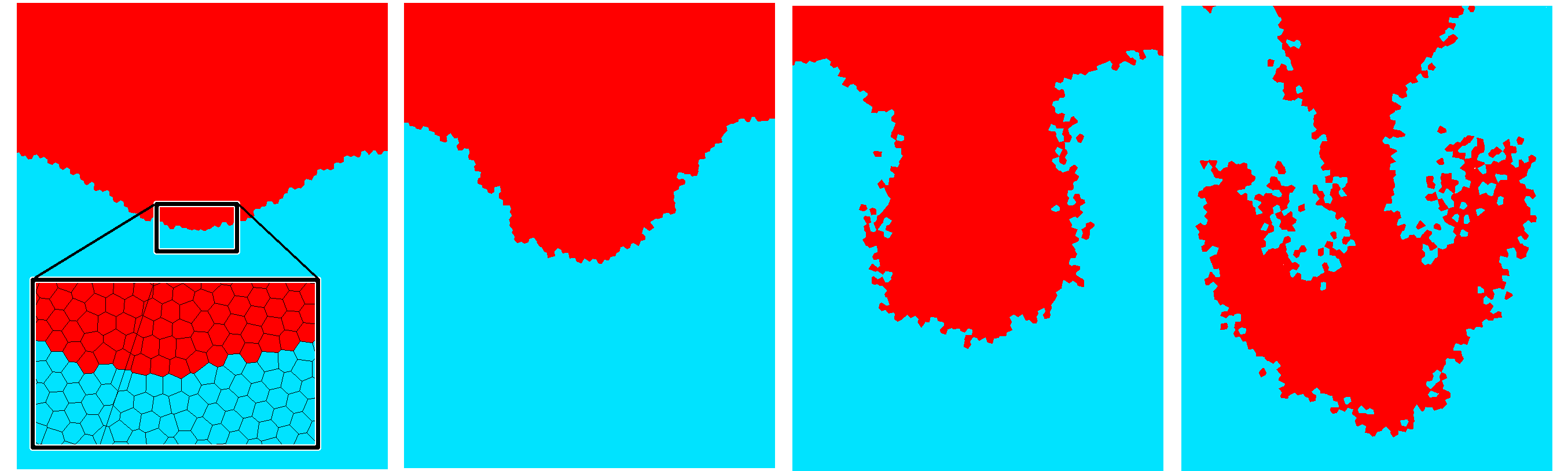}
      }
      \caption{An application of optimal transport to fluid
        simulation: numerical simulation of the Taylor-Rayleigh
        instability using the Gallouet-M\'erigot scheme.}
      \label{fig:instability_2D}
\end{figure*}  

\begin{figure*}
      \centerline{
         \includegraphics[width=\textwidth]{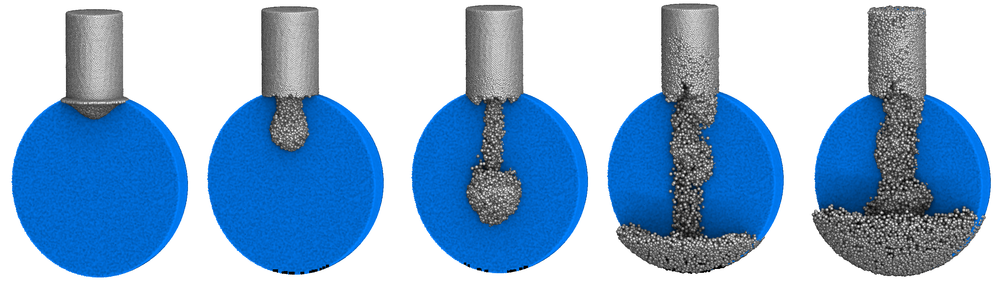}
      }
      \caption{Numerical simulation of an incompressible bi-phasic flow in a bottle.}
      \label{fig:bottle}
\end{figure*}  

\begin{figure}
      \centerline{
         \includegraphics[width=0.7\columnwidth]{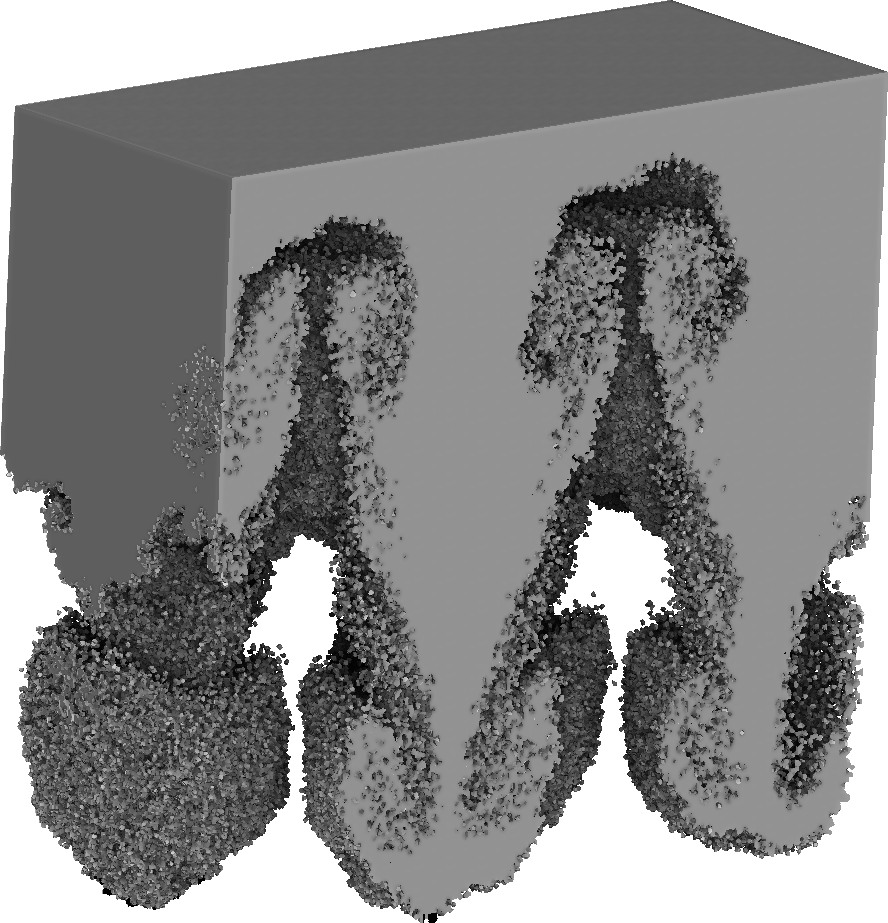}
      }
      \vspace{20mm}
      \centerline{
         \includegraphics[width=0.7\columnwidth]{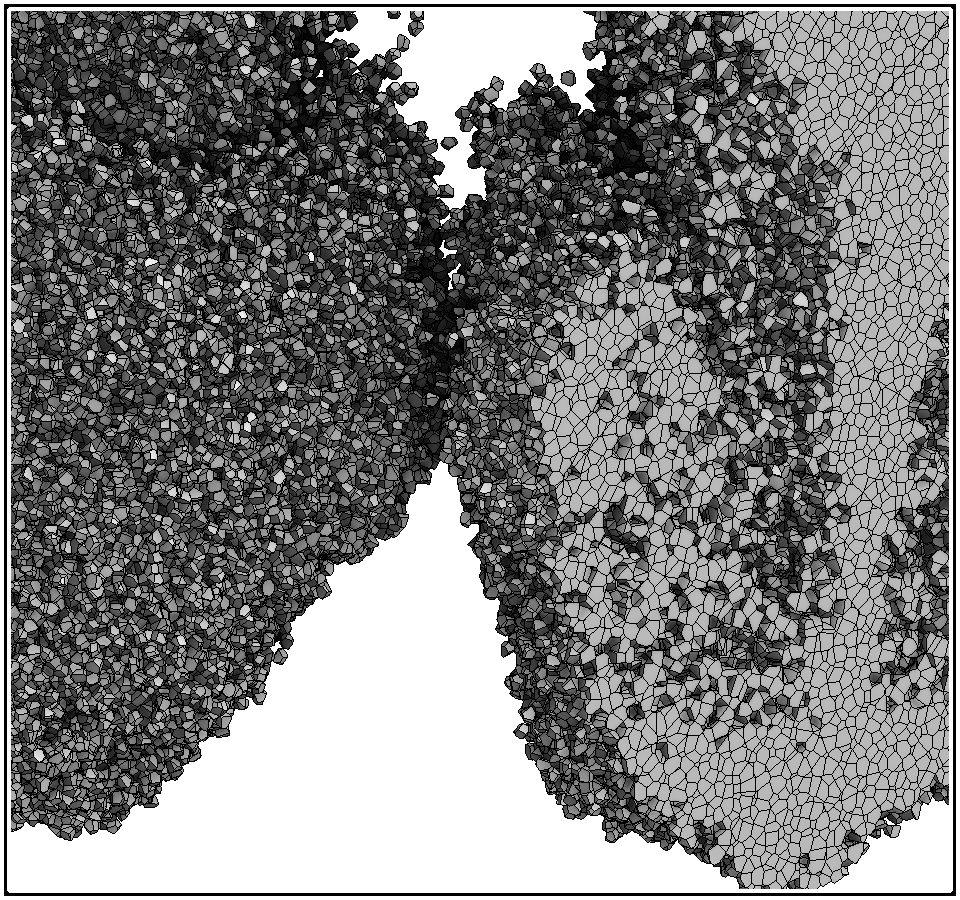}
      }
      \caption{Top: Numerical simulation of the Taylor-Rayleigh instability using a 3D version
        of the Gallouet-M\'erigot scheme, with a cross-section that reveals the internal
        structure of the vortices. Bottom: a closeup that shows the interface between the two fluids, represented by the Laguerre
        facets that bound two Laguerre cells of different fluid elements.}
      \label{fig:instability_3D}
\end{figure}

Figure \ref{fig:exemples_2D} shows some examples of transport in
2D, between a uniform density and a pointset (A). Each Laguerre
cell has the same area. Then we consider the same pointset, but this
time with a density $u(x,y) = 10(1 + \sin(2 \pi x)\sin(2 \pi y))$ (image B).
We obtain a completely different
Laguerre diagram. Each cell of this diagram has the same value for the integrated
density $u$. (C): the coefficients of the gradient and the Hessian of $F$ computed
in the previous section involve integrals of the density $u$ over the Laguerre
cells and over their boundaries. The density $u$ is supported by a triangulated mesh,
and linearly interpolated on the triangles. The integrals of $u$ over the Laguerre
cells and their boundaries are evaluated in closed form, by computing the intersections
between the Laguerre cells and the triangles of the mesh that supports $u$. (D): the
same algorithm can compute the optimal transport between a measure supported by a 3D
surface and a pointset in 3D.

Figure \ref{fig:exemple_3D} shows two examples of volume deformation by optimal
transport. The same algorithm is used, but this time with 3D Laguerre diagrams, and
by computing intersections between the Laguerre cells and a tetrahedral mesh that
supports the source density $u$. The intermediary steps of the animation are generated
by linear interpolation of the positions (Mc. Cann's interpolation).
Figure \ref{fig:SphereCube} demonstrates a more challenging configuration, 
computing transport between a sphere (surface) and a cube (volume).
The sphere is approximated with 10 million Dirac masses. As can be
seen, transport has a singularity that resembles the medial axis of
the cube (rightmost figure). 

Figure \ref{fig:instability_2D} demonstrates an application in computational
fluid dynamics. A heavy incompressible fluid (in red) is placed on top of a
lighter incompressible fluid (in blue). Both fluids tend to exchange their
positions, due to gravity, but incompressibility is an obstacle to their movement.
As a result, vortices are created, and they become faster and faster (Taylor-Rayleigh
instability). The numerical simulation method that we used here
\cite{Gallouet2017} directly computes the trajectories of particles
(Lagrangian coordinates) while taking into account the incompressibility
constraint, which is in general unnatural in Lagrangian coordinates. By providing
a means of controlling the volumes of the Laguerre cells, optimal transport appears
here as an easy way of enforcing incompressibility. Using some efficient geometric
algorithms for the Laguerre diagram and its intersections \cite{journals/M2AN/LevyNAL15},
the Gallouet-M\'erigot numerical scheme for incompressible Euler scheme can be also applied in 3D
to simulate the behavior of non-mixing incompressible fluids (Figure \ref{fig:bottle}).
The 3D version of the Taylor-Rayleigh instability is shown in Figure \ref{fig:instability_3D},
with 10 million Laguerre cells. More complicated vortices are generated (shown here in cross-section).\\

The applications in computer animation and computational physics seem to be
promising, because the semi-discrete algorithm behaves very well in practice.
It is now reasonable to design simulation algorithms that solve a transport
problem in each timestep (as our early computational fluid dynamics of the previous
paragraph do). With our optimized implementation that uses a multicore processor
for the geometric part of the algorithm and a GPU for solving the linear systems,
the semi-discrete algorithm takes no more than a few tens of seconds to solve a problem
with 10 millions unknown. This opens the door to numerical solution mechanisms for
difficult problems. For instance, in astrophysics, the Early Universe Reconstruction
problem  \cite{EUR} consists in going \entreguillemets{backward in time} from
observation data on the repartition of galaxy clusters, to
\entreguillemets{play the Big-Bang movie} backward. Our numerical experiments
tend to confirm Brenier's point of view that he expressed in the 2000's, that
semi-discrete optimal transport can be an efficient way of solving this problem. \\

To ensure that our results are reproducible, the source-code associated with the numerical
solution mechanism used in all these
experiments is available in the EXPLORAGRAM component of the GEOGRAM
programming library\footnote{\url{http://alice.loria.fr/software/geogram/doc/html/index.html}}.

\section*{Acknowledgments}
This research is supported by EXPLORAGRAM (Inria Exploratory Research Grant). The authors wish
to thank Quentin M\'erigot, Yann Brenier, Jean-David Benamou, Nicolas Bonneel and L\'ena\"{\i}c Chizat
for many discussions.

\bibliographystyle{alpha}
\bibliography{ot}

\end{document}